\documentclass[11pt]{amsart}
\usepackage{amssymb}
\usepackage{graphicx}
\usepackage{xcolor} 
\usepackage{tensor}
\usepackage{fullpage} 
\usepackage{amsmath}
\usepackage{yhmath}
\usepackage{amsthm}
\usepackage{verbatim}
\usepackage{hyperref}
\usepackage{array} 
\usepackage{enumitem}
\setlist[enumerate]{leftmargin=1.8em}
\setlist[itemize]{leftmargin=1.8em}
\usepackage[notref, notcite]{showkeys} 
\usepackage{seqsplit,mathtools}

\definecolor{green}{rgb}{0,0.8,0} 

\newtheorem{theorem}{Theorem}[section]

\newtheorem{lemma}[theorem]{Lemma}

\theoremstyle{definition}

\theoremstyle{remark}
\newtheorem{remark}[theorem]{Remark}

\numberwithin{equation}{section}
\newcommand{\nrm}[1]{\Vert#1\Vert}

\newcommand{\nnrm}[1]{{\vert\kern-0.25ex\vert\kern-0.25ex\vert #1 
		\vert\kern-0.25ex\vert\kern-0.25ex\vert}}

\newcommand{\sgn}{{\mathrm{sgn}}\,}
\newcommand{\supp}{{\mathrm{supp}}\,}

\newcommand{\lap}{\Delta}

\newcommand{\rd}{\partial}
\newcommand{\nb}{\nabla}

\newcommand{\Gmm}{\Gamma}

\newcommand{\omg}{\omega}

\newcommand{\bbR}{\mathbb R}

\definecolor{purple}{rgb}{0.65, 0, 1}
\definecolor{orange}{rgb}{1,.5,0}

\begin{document}

	\title{Well-posedness for Vlasov--Poisson on Low Regularity $H^{s,p}$ Spaces in All Dimensions.}
	
	\author{Sangwook Tae}
	\address{Department of Mathematical Sciences and RIM, Seoul National University, 1 Gwanak-ro, Gwanak-gu, Seoul 08826,  Republic of Korea.}
	\email{swtae00@snu.ac.kr}

	\date{\today}

	\renewcommand{\thefootnote}{\fnsymbol{footnote}}
	\renewcommand{\thefootnote}{\arabic{footnote}}

	\maketitle

	\begin{abstract}
		We consider the Vlasov--Poisson equation on $\mathbb{R}^d \times \mathbb{R}^d$ for any dimension. For initial distribution $f_{0}$ having compact support in $v$ and belonging to $H^{s,p}(\mathbb{R}^d \times \mathbb{R}^d)$, we prove local well-posedness for $s>\frac{d}{p}-\frac{1}{2p}$ and $p\in[2, \infty)$. We proved this by using two main ingredients, which is the averaging property for the density $\rho$ and the Schauder-Tychonoff fixed point theorem. It seems like this is the first application of the Schauder-Tychonoff fixed point theorem in showing existence of solution for evolutionary PDEs.

	\end{abstract}
	
	\section{Introduction}
	In this paper, we study well-posedness of the Vlasov--Poisson equation in $L^{p}$-based Sobolev spaces with relatively low regularity. The equation describes evolution of a distribution $f(t,x,v) : \bbR \times \bbR^d \times \bbR^d \to \bbR $: 
	\begin{equation}\label{eq:VP} \tag{VP}
		\left\{
		\begin{aligned}
			\rd_{t} f + v \cdot \nb_{x} f - \nb_{x} U \cdot \nb_{v} f = 0, & \\
			\lap_{x} U = \pm\rho = \pm\int_{\bbR^{d}} f dv. & 
		\end{aligned}
		\right.
	\end{equation} 	$U(t,x)$ is explicitly given by \begin{equation*}
	\begin{split}
		U(t,x) = \int_{\bbR^{d}}  \Gmm_{d}(x-y)  \rho(t,y) dy,
	\end{split}
	\end{equation*}
	where $\Gmm_d$ is given by
	\begin{equation*}
		\Gmm_d(x)=
		\begin{cases}
			\frac{1}{d(2-d)\omg_{d}} |x|^{2-d}, \qquad \, &d \neq 2 \\
			\frac{1}{2\pi}\log(\vert x\vert), \qquad \qquad  &d=2
		\end{cases}
	\end{equation*}
	Here, $\omg_{d}$ is the volume of the unit $d$-ball. Moreover, although we consider $+$ sign for the potential in \eqref{eq:VP}, the result is valid for both signs: the positive sign corresponds to collision-less plasma dynamics and the negative to astrophysical situations, most notably galaxy dynamics (\cite{Glassey96,Horst,Horst2}).

	 Our result gives local well-posedness of solutions to \eqref{eq:VP} in $H^{s,p}$ spaces not in the classical regime, which holds for all $d$. In this paper, we define $H^{s, p}(\mathbb{R}^d)=\{ f\in L^p(\mathbb{R}^d): \mathcal{F}^{-1}\langle \xi \rangle^s \mathcal{F}f\in L^p(\mathbb{R}^d)  \}$ where $\mathcal{F}$ denotes the Fourier transform.
	
	\begin{theorem}\label{thm:main}
		The Vlasov--Poisson equation \eqref{eq:VP} is locally well-posed in $(H^{s,p} \cap L^{1})(\bbR^{d} \times \bbR^{d})$ with compact support in $v$ for $s> \frac{d}{p}-\frac{1}{2p}$ and $p\geq 2$. That is, with initial data $f_{0} \in (H^{s,p} \cap L^{1})(\bbR^{d} \times \bbR^{d})$ satisfying $f_{0}(x,v)  = 0$ whenever $|v|>Q/2$ for some $Q>0$, there exist $T>0$ depending on $\nrm{f_{0}}_{H^{s,p}\cap L^{1}}, Q$, and a unique solution $f \in C([-T,T]; (H^{s,p} \cap L^{1})(\bbR^{d} \times \bbR^{d}))$ to \eqref{eq:VP} with $f(t=0) = f_{0}$ satisfying $f(t,x,v) = 0$ whenever $|v| > Q$. 
	\end{theorem}
	
	\begin{remark}
		We give a few remarks regarding the statement of the theorem. 
		
		\begin{itemize}
			\item The phrase ``not in the classical regime" means that the solution $f$ does not have to be continuous (or even in $L^\infty$ or $L^q$ for sufficiently large $q$). This was possible due to averaging effect for the density $\rho$. Indeed, the averaging Lemma tells us that $\rho$ gains $\frac{1}{2p}$ more regularity compared to $f$. This enables us to define the flow map of the characteristic vector field, which was the key to success.
			\item The compact support assumption in $v$ can be relaxed to the
			case when the solution decays sufficiently fast (e.g., exponential decay) in $v$.
			\item The main achievement is that the dimensional restriction is relaxed compared to the result in \cite{Prev}. This was possible by different approaches on constructing the solution.
		\end{itemize}
		
	\end{remark}

	\subsection{Organization of the paper}

	The rest of the paper is organized as follows. First, we provide some previous results on well posedness for Vlasov-Poisson equation and also motivations in our proofs. Then, we prove for $d\geq 2$, the existence in \S \ref{Existence} and prove the uniqueness and continuity in \S \ref{unique}. In \S \ref{1D}, we clarify the necessary modifications for the proof of the case $d=1$. Finally, in \S \ref{Para}, we present the detailed steps of a priori estimates using paradifferential calculus.  

	\bigskip

	\noindent \textbf{Acknowledgments}. The author would like to express his gratitude and thanks to Professor In-Jee Jeong for his guidance throughout this project.

	\subsection{Previous works on well-posedness}
		
	We now review various works on well-posedness of \eqref{eq:VP}  on the whole space case $\mathbb{R}^d \times \mathbb{R}^d$. A more extensive list of references, as well as history, can be found in the textbooks \cite{Glassey96,Rein07,BGP}. 
		
	First, Horst \cite{Horst} proves local well-posedness for $C^1$ data in the whole space  with any $d\ge 1$, with global well-posedness for dimensions 1 and 2. Global well-posedness for classical solutions in $d = 3$ which is much more harder, was obtained in \cite{K.P.,P.L.,J.S.}. For $d \ge 4$, it was shown that classical solutions can blow up in finite time (\cite{Horst2}). Recently, Chen and He (\cite{W.F.S._1,W.F.S._2,Besov}) give local well-posedness in (weighted in $v$) Sobolev, Bessel potential spaces, and Besov spaces. Their Besov local well-posedness works in $B^{s}_{p, r}(\mathbb{R}^d)$, where $s>\frac{d}{p}+1$, $1<p<\frac{d}{2}$, $d\geq 3$, and $1<r<\infty$. 
		
	On the other hand, it has been well-known since \cite{Ars} that global existence of weak solutions holds with little regularity assumptions on the data, see also \cite{Lagrangian, C.P.,P.L.} for the theory of weak solutions.  Regarding the question of uniqueness for weak solutions, Loeper \cite{Loeper} gave a powerful criterion which states that boundedness of $\rho$ in spacetime garantees uniqueness. This uniqueness criterion was strengthened in several directions. Among those, \cite{BV} showed that for $d = 1$, BV regularity ensures well-posedness. Moreover, \cite{Uniqueness,Yudovich} extended the uniqueness criterion to $\rho$ belonging to the so-called Yudovich spaces, which allows the $L^p$ norm of $\rho$ to grow at some rate as $p\to\infty$.
	
	Finally, in \cite{Prev}, the well posedness in $H^s$ for $s>\frac{d}{2}-\frac{1}{4}$ was proven for $d\geq 3$. This is the first result for proving well posedness outside the classical regime, for best of our knowledge. This result can be seen as a special case for our result, which generalized to $H^{s, p}$ for $p\geq 2$ and also extended to the case $d=1, 2$. Despite this fact, we emphasize that the strategy of the proof differs significantly, which is one of the reason for success to covering the case $d=1, 2$. 
	
	\subsection{Difficulties}
	
	\subsubsection{Issues in low dimensions $d = 1, 2$}
	
	We have mentioned above that the $H^s$ well posedness was accomplished for the case $d\geq 3$ in \cite{Prev}. The reason for the failure in the case $d=1, 2$ was twofold. First, when performing a priori estimates, the elementary analysis on the Fourier side required the following bound $$\Vert \vert\nabla_x\vert^{s-1}\vert\nabla_v\vert  f \Vert_{L^2}\leq C\Vert f \Vert_{H^s_{x,v}}, $$ which is only true for $s\geq1$. This fact became an obstacle for proving the case $d=1, 2$ because of the condition $s>\frac{d}{2}-\frac{1}{4}$. Second, the Sobolev embedding $\dot{W}^{1, \frac{2d}{d+2}}\hookrightarrow L^2$ was strongly used to control the $L^2$ norm of the electric field, which does not hold for $d=1, 2$. Actually, for the case $d=1$, the situation was fundamentally severe since the electric field actually has no decay in space, as one can observe.
	
	In this work, we have circumvented these difficulties by using more refined paraproduct techniques. It allowed us to obtain a similar a priori estimates without using the problematic inequality $\Vert \vert\nabla_x\vert^{s-1}\vert\nabla_v\vert  f \Vert_{L^2}\leq C\Vert f \Vert_{H^s_{x,v}}$ and at the same time, allowed us to generalize the result for arbitrary $p\geq 2$. This method also opened up the path to overcome the second issue by the following procedure. For the case $d=2$, the only problem was the invalidity of the Sobolev embedding $\dot{W}^{1, 1}\hookrightarrow L^2$. It corresponds to the endpoint case, which is the reason for the failure. We slightly adujusted the integrability index to overcome this problem. 
	
	For the case $d=1$, as mentioned before, the absence of decay for the electric field was problematic. Fortunately, the electric field was bounded in space, which allowed us to adjust the paraproduct estimates to handle this case also. It would have been impossible if we have just used elementary arguments on the Fourier side, since the Fourier transform of electric field would have not belong to a function, but rather a distribution.

	\subsubsection{Motivations for applying fixed point theorem in the proof}
	
	As said earlier, the key mechanism for the well posedness is the averaging effect for the density field $\rho$, which makes the flow map of the characteristic vector field well defined, thus allowing us to view the solution as a composition of the initial data and the inverse flow map. Based on this philosophy, we first tried to construct the solution in the following iteration procedure.
	
	Given the initial ansatz $f_1$ defined on some time interval $[0, T]$, we calculate the corresponding density field $\rho_1$, generate the characteristic vector field to obtain the flow map $\Phi_1$, and define $f_2=f_0\circ\Phi_1^{-1}$ and view this as a new ansatz to iterate. This method seems natural, but some difficulties occur. First, starting with arbitrary $f_1$ will not work since there is no guarantee that the corresponding $\rho_1$ has enough regularity in general. Also, it seems hard to get the precise estimates for the flow map in terms of the initial ansatz. To go around this difficulties, we chose an alternative route.
	
	The observation for this route is the following. First, notice that the flow is completely determined by $\rho$, which does not require full information about $f$. Once we obtain the flow, we can recover $f$ by composing initial $f_0$ to the inverse of the flow. Therefore, we see that we can iterate with $\rho$ instead of $f$. The great advantage of this method is that we can encode the gained regularity of $\rho$ directly, which was extremely hard when iterating with $f$. 
	
	However, there is one more difficulty. The estimates we have in our hand is the bound of certain quantities, so when we iterate to obtain sequence of $\rho$'s, what we obtain is the convergence through subsequence, which comes from compactness of such sequence. But we immediately see that the limiting value obtained has no guarantee to satisfy the original equation. This is because of the following reason. $\rho$ obtained in certain step was obtained by flow map generated by $\rho$ in the previous step, so such information is completely lost when passing through subsequence. Therefore, we need the whole sequence to converge, not through the subsequence.
	
	To circumvent this problem, we decided to use some fixed point theorem instead of iteration. One of the ones used in proving existence is Banach fixed point theorem, which required the contraction property of mapping. In our case, such estimate was not available. We needed the theorem with weaker requirements. The other alternative route was to use the Schauder fixed point theorem, which requires the mapping to be continuous and the domain and range to be compact, convex in Banach space. But showing the compactness in Banach space topology is extremely hard. The compactness is usually easy to show in the weak topology for separable Banach space, due to Banach Alaoglu theorem. In our case, since we have boundedness property in hand, this is the case, actually. However, this topology is not Banach space topology, but only locally convex topological vector space. But most of the fixed point theorem is for Banach space, so these were not available for our case.
	
	Fortunately, there was the fixed point that suites for our case, which was Schauder-Tychonoff fixed point theorem. The statement is given in the next section, and as one can see, the requirement is the continuity for the mapping and the convex and compactness of domain and range. The latter condition is relatively easy to satisfy, but the continuity was not trivial. However, we managed to show the continuity, which enabled us to construct the solution explicitly. 
	
	One final thing we want to mention is that to the best of our knowledge, this is the first case to construct the solution for evolutionary PDE using Schauder-Tychonoff fixed point theorem. We think that the reason for this is due to the difficulty for proving the continuity in weak topology. Therefore, we claim that this paper showed some new strategy for constructing the solution for evolutionary PDE. 
	\section{Existence}\label{Existence}
	In this section, we are to construct the solution. To do this, we use fixed point method. The important fact is that we are finding a fixed point of $\rho$ instead of $f$.\\
	The fixed point theorem we will use is the following Schauder-Tychonoff fixed-point Theorem\cite{Rudin}:
	\begin{theorem}
		\label{FP}
		Let $X$ be a locally convex topological vector space and let $K$ be a nonempty, compact, convex subset of $X$. Then, any continuous mapping $F:K\to K$ has a fixed point.
	\end{theorem} 
	This theorem has the power that it applies to any locally convex topological vector spaces and the compactness of the map is not required. However, it has some limitations that it is even hard to prove continuity and compactness in general topological vector spaces. Despite this fact, we have managed to derive the continuity of the mapping defined below.\\
	
	We construct the settings that will be applied to Schauder-Tychonoff fixed-point Theorem. We first let
	\begin{flalign*}
		V=&\{\rho:[0, T]\times \mathbb{R}^d\to \mathbb{R}\,\vert\, \rho\in L^{1+\epsilon p}([0, T], H^{s+\frac{1}{2p}-\epsilon, p}(\mathbb{R}^d)),\;\nabla_x\Delta_x^{-1}\rho\in L^{p}([0, T], H^{s', p+\epsilon'}(\mathbb{R}^d))  \\&\;\text{and}\;  \partial_t \nabla_x\Delta_x^{-1}\rho\in L^{p+\epsilon'}([0, T], L^{p+\epsilon'}(\mathbb{R}^d))\}.
	\end{flalign*} 
	Here, $\epsilon, \epsilon'>0$ is sufficiently small and  $s'=s+\frac{1}{2p}-\epsilon-\frac{\epsilon'd}{p(p+\epsilon')}$ . This space is a reflexive Banach Space. We set $X=V$ with the weak topology. This is the locally convex topological vector space that we will use. Now, for the compact subset, we set
	$$K=\overline{B_R(0)}\cap \left\{\rho\in V: \rho(0, x)=\int_{\mathbb{R}^d}f_0(x, v)dv \right\} ,$$ where $\overline{B_R(0)}$ is a closed ball of $V$ with radius $R$. By Banach-Alaoglu theorem, $\overline{B_R(0)}$ is compact in the weak topology. The set $\left\{\rho\in V: \rho(0, x)=\int_{\mathbb{R}^d}f_0(x, v)dv \right\}$ is a closed subset of $V$ in the strong topology by the continuity of the trace operator. Since it is convex, by Mazur's theorem, it is also weakly closed. The only concern is the non-emptiness of this set. We shall discuss this issue a bit later since it requires a priori estimates obtained below. Being a closed subset of a compact set, $K$ is also compact. It is convex since it is an intersection of two compact sets. It can be made nonempty by taking $R$ sufficiently large. Finally, the functional we will use in this setting is the following:
	\begin{flalign*}
		G[\rho](t, x):=\int_{\mathbb{R}^d} (f_0\circ (\Phi_t^\rho)^{-1})(t, x, v)dv
	\end{flalign*}
	where $f_0\in H^{s, p}(\mathbb{R}^d\times\mathbb{R}^d)$ is the initial condition and $\Phi^\rho$ is the flow map of the nonautonomous vector field $(v, -\nabla_x\Delta_x^{-1}\rho(x, t))$, which will be shown to be well defined. We will adjust $T$ (but not $R$) to show that $G:K\to K$. So, we have two steps: first, we should show that $G$ is well defined and maps $K$ into itself and second, we should show that $G$ is continuous in this topology.
	
	Before starting the proof, let us remark that we are always denoting $C$ to be a constant independent of $R$ and $T$, but may depend on other factors. 
	\subsection{Well Definedness and Self Mapping Property of $G$}
	Let $\rho\in K$. First, we consider the well definedness of the flow. This is actually a consequence of the Sobolev embedding theorem. Recall that our vector field is $(v, -\nabla_x\Delta_x^{-1}\rho(x, t))$. By sobolev embedding, we have $$H^{s+\frac{1}{2p}-\epsilon, p}(\mathbb{R}^d)\hookrightarrow H^{s', p+\epsilon'}(\mathbb{R}^d)$$
	where $s'=s+\frac{1}{2p}-\epsilon-\frac{\epsilon'd}{p(p+\epsilon')}$, which implies that
	\begin{flalign*}
		\rho\in L^1([0, T], H^{s', p+\epsilon'}(\mathbb{R}^d)).
	\end{flalign*}
	Also, we have that $$\nabla_x\Delta_x^{-1}\rho\in L^1([0, T], L^{p+\epsilon'}(\mathbb{R}^d)),$$
	which can be seen as follows.\\
	
	Since $f_0(x,v)=0$ for $\vert v\vert< Q$, we have that 
	\begin{flalign*}
		\Vert \rho(0) \Vert_{L^p_x}\leq \Vert Q^{d-\frac{d}{p}}\Vert f_0(x, \cdot)\Vert_{L^p_v}  \Vert_{L^p_x}\leq Q^{d-\frac{d}{p}}\Vert f_0 \Vert_{L^p_{x,v}}. 
	\end{flalign*}
	Also, we have $$\Vert \rho(0) \Vert_{L^1_x}\leq \Vert f_0\Vert_{L^1_{x,v}}$$
	This implies that $$\Vert \rho(0) \Vert_{L^{p'}_x}\leq \Vert \rho(0) \Vert_{L^p_x}^\frac{1-\frac{1}{p'}}{1-\frac{1}{p}} \Vert \rho(0) \Vert_{L^1_x}^{\frac{\frac{1}{p}+\frac{1}{p'}}{1-\frac{1}{p}}}\leq \left(  Q^{d-\frac{d}{p}}\Vert f_0 \Vert_{L^p_{x,v}} \right)^\frac{1-\frac{1}{p'}}{1-\frac{1}{p}}\Vert f_0 \Vert_{L^1_{x,v}}^{\frac{\frac{1}{p}+\frac{1}{p'}}{1-\frac{1}{p}}}=C.$$
	where $p'=\frac{(p+\epsilon')d}{p+\epsilon'+d}$ and $C$ is independent of $T$ and $R$. Now by the Sobolev embedding
	\begin{flalign*}
		L^{p'}\hookrightarrow \dot{H}^{-1, p+\epsilon'}(\mathbb{R}^d),
	\end{flalign*}
	we have
	\begin{flalign*}
		\Vert \nabla_x\Delta_x^{-1}\rho(0) \Vert_{L^{p+\epsilon'}_x} \leq C.
	\end{flalign*}
	This and the assumption $\partial_t \nabla_x\Delta_x^{-1}\rho\in L^{p+\epsilon'}([0, T], L^{p+\epsilon'}(\mathbb{R}^d))$ implies that we have
	\begin{flalign*}
		\nabla_x\Delta_x^{-1}\rho\in C([0, T], L^{p+\epsilon'}(\mathbb{R}^d)).	
	\end{flalign*}
	and indeed, we have
	\begin{flalign*}
		\Vert \nabla_x\Delta_x^{-1}\rho(t) \Vert_{L_x^{p+\epsilon'}} \leq C+CRT^{1-\frac{1}{p+\epsilon'}}
	\end{flalign*}
	We therefore have that $$\nabla_x\Delta_x^{-1}\rho\in L^1([0, T], H^{s'+1, p+\epsilon'}(\mathbb{R}^d)),$$ 
	which implies that 
	\begin{flalign*}
		\nabla_x\Delta_x^{-1}\rho\in L^1([0, T], C^1(\mathbb{R}^d))
	\end{flalign*}
	by the Sobolev embedding 
	\begin{flalign*}
		H^{s'+1, p+\epsilon'}(\mathbb{R}^d)\hookrightarrow C^1(\mathbb{R}^d).
	\end{flalign*}
	Thus, by the standard Cauchy-Lipschitz theory, we see that the flow map is well defined. Let us write $\Phi^\rho(t, x, v)=(X^\rho(t, x, v), V^\rho(t, x, v))$. Then, the following holds:
	\begin{flalign*}
		V^\rho(t, x, v)=v-\int_{0}^t\nabla_x\Delta_x^{-1}\rho(s,X^\rho(s,x,v))ds.
	\end{flalign*}
	Thus, by the triangle inequality, we have
	\begin{flalign*}
		\left\vert V^\rho(t, x, v)\right\vert&=\vert v\vert +\int_{0}^t \left\Vert\nabla_x\Delta_x^{-1}\rho(s)\right\Vert_{C(\mathbb{R}^d)} ds \leq \vert v\vert +C \int_{0}^t \left\Vert\nabla_x\Delta_x^{-1}\rho(s)\right\Vert_{H^{s',p+\epsilon'}(\mathbb{R}^d)} ds\\
		&\leq \vert v\vert +CT^{\frac{p\epsilon}{1+p\epsilon}}\Vert \nabla_x\Delta_x^{-1}\rho \Vert_{L^{1+p\epsilon}([0, T], H^{s',p+\epsilon'}(\mathbb{R}^d))}\\
		&\leq \vert v\vert +CT^\frac{p\epsilon}{1+p\epsilon}\Vert \rho \Vert_{L^{1+p\epsilon}([0, T], H^{s',p+\epsilon'}(\mathbb{R}^d))}+CT^\frac{p\epsilon}{1+p\epsilon}\Vert \nabla_x\Delta_x^{-1}\rho \Vert_{L^{1+p\epsilon}([0, T], L^{p+\epsilon'}(\mathbb{R}^d))}\\
		&\leq \vert v\vert +CT^{\frac{p\epsilon}{1+p\epsilon}}\Vert \rho \Vert_{L^{1+p\epsilon}([0, T], H^{s+\frac{1}{2p}-\epsilon,p}(\mathbb{R}^d))}+CT\Vert \nabla_x\Delta_x^{-1}\rho \Vert_{L^\infty([0, T], L^{p+\epsilon'}(\mathbb{R}^d))}\\
		&\leq \vert v\vert +CRT^{\frac{p\epsilon}{1+p\epsilon}}+CT(1+T^{1-\frac{1}{p+\epsilon'}}R).
	\end{flalign*}
	Thus, if we initially have $\vert v\vert \leq \frac{Q}{2}$ and choose $T>0$ sufficiently small, then we have that $\vert V^\rho(t, x, v) \vert\leq Q $. Therefore, we see that the function 
	$$f^\rho(t, x, v)=(f_0\circ (\Phi_t^\rho)^{-1})(x, v)$$ is well defined and has the property that $f^\rho=0$ for $\vert v\vert \geq Q$. This property is crucial for applying the averaging lemma later.
	
	Next, we need to show that $f^\rho$ is in $L^\infty([0, T], H^{s,p}(\mathbb{R}^d\times\mathbb{R}^d))$. Unfortunately, this cannot be achieved easily, since our flow map is guaranteed only to be in $C^1$, not in $C^{\lfloor s\rfloor+1}$. However, we have the information that $\Phi^\rho$ is a flow map of $L_1^t H^{s',p+\epsilon'}$ vector field, and this is the key to success. To exploit this fact, we have to go back to the transport equation that is satisfied by the function $f^\rho$:
	\begin{flalign*}
		\partial_t f^\rho+v\cdot \nabla_x f^\rho-\nabla_x\Delta_x^{-1}\rho\cdot \nabla_v f^\rho=0
	\end{flalign*} 
	To estimate the $H^{s, p}$ norm, we have to look into the equation governing the $\Lambda_{x, v}^sf^\rho$:
	\begin{flalign*}
		\partial_t \Lambda_{x,v}^sf^\rho+\Lambda_{x,v}^s[(v,-\nabla_x\Delta_x^{-1}\rho)\cdot\nabla_{x,v}f^\rho]=0.
	\end{flalign*}
	We transform this equation into the commutator form:
	\begin{flalign*}
		\partial_t \Lambda_{x,v}^sf^\rho+(v,-\nabla_x\Delta_x^{-1}\rho)\cdot\nabla_{x,v}(\Lambda_{x,v}f^\rho)=-[\Lambda_{x,v}^s,(v,-\nabla_x\Delta_x^{-1}\rho)\cdot \nabla_{x,v}]f^\rho,
	\end{flalign*}
	where $\Lambda_{x,v}=(1-\Delta_{x,v})^\frac{1}{2}$. Multiplying both equation by $\vert \Lambda_{x,v}^s f^\rho\vert^{p-2}\overline{\Lambda_{x,v}^s f^\rho}$ and integrating in $(x,v)$ gives\footnote{Note that our vector field is divergence free, so the transport term cancels out when performing integration by parts.}
	\begin{flalign*}
		\frac{1}{p}\frac{\partial}{\partial t}\Vert \Lambda_{x,v}^s f^\rho\Vert_{L^p}^p&=-\int \vert \Lambda_{x,v}^s f^\rho\vert^{p-2}\overline{\Lambda_{x,v}^s f^\rho} [\Lambda_{x,v}^s,(v,-\nabla_x\Delta_x^{-1}\rho)\cdot \nabla_{x,v}]f^\rho dxdv \\
		&\leq \Vert \Lambda_{x,v}^s f^\rho \Vert_{L^p}^{p-1} \Vert [\Lambda_{x,v}^s,(v,-\nabla_x\Delta_x^{-1}\rho)\cdot \nabla_{x,v}]f^\rho \Vert_{L^p}
	\end{flalign*}
	and thus we obtain
	\begin{flalign*}
		\frac{\partial}{\partial t}\Vert \Lambda_{x,v}^s f^\rho\Vert_{L^p} \leq \Vert [\Lambda_{x,v}^s,(v,-\nabla_x\Delta_x^{-1}\rho)\cdot \nabla_{x,v}]f^\rho \Vert_{L^p}.
	\end{flalign*}
	We first estimate the simpler term $[\Lambda_{x,v}^s,v\cdot\nabla_x]f^\rho$. We perform a Fourier trasform of this term to obtain:
	\begin{flalign*}
		\widehat{[\Lambda_{x,v}^s,v\cdot\nabla_x]f^\rho}= \langle (\xi, \eta) \rangle^s (\xi \cdot \nabla_\eta)\widehat{f^\rho}-\xi\cdot\nabla_\eta(\langle (\xi, \eta) \rangle^s \widehat{f^\rho}) = -(\xi\cdot \nabla_\eta\langle (\xi, \eta) \rangle^{s})\widehat{f^\rho}=-s(\xi\cdot\eta) \langle (\xi, \eta) \rangle^{s-2} \widehat{f^\rho}.
	\end{flalign*}
	Thus, the result is just the fourier multiplier operator applied to $f^\rho$. This multiplier can be shown to satisfy the Mikhlin multiplier theorem, so we obtain
	\begin{flalign*}
		\Vert [\Lambda_{x,v}^s,v\cdot\nabla_x]f^\rho \Vert_{L^p(\mathbb{R}^d\times\mathbb{R}^d)} \leq C\Vert f^\rho \Vert_{H^{s,p}(\mathbb{R}^d\times\mathbb{R}^d)}.
	\end{flalign*} 
	The second term $[\Lambda_{x,v}^s, \nabla_x\Delta_x^{-1}\rho\cdot \nabla_v]f^\rho$ is much trickier to estimate, and can be done using paraproduct estimates.
	
	We will show in Section \ref{Para} that 
	\begin{flalign*}
		&\frac{\partial}{\partial t}\Vert  f^\rho\Vert_{H^{s,p}(\mathbb{R}^d\times\mathbb{R}^d)} \leq C\left(1+\Vert \nabla_x\Delta_x^{-1}\rho \Vert_{H^{s'+1,p+\epsilon'}(\mathbb{R}^d)} \right)\Vert f^\rho \Vert_{H^{s,p}(\mathbb{R}^d\times\mathbb{R}^d)}\\
		&\leq C\left(1+\Vert \rho \Vert_{H^{s+\frac{1}{2p}-\epsilon,p}(\mathbb{R}^d)}+ \Vert \nabla_x\Delta_x^{-1}\rho \Vert_{L^{p+\epsilon'}(\mathbb{R}^d)}\right)\Vert f^\rho \Vert_{H^{s,p}(\mathbb{R}^d\times\mathbb{R}^d)}
	\end{flalign*}
	and therefore, by using Gronwall's inequality, we have
	\begin{flalign*}
		\Vert f^\rho(t)\Vert_{H^{s,p}_{x,v}}&\leq \Vert f_0\Vert_{H^{s,p}_{x,v}}\exp\left(CT+CT^{\frac{p\epsilon}{1+p\epsilon}}\Vert \rho \Vert_{L^{1+p\epsilon}([0, T],H^{s+\frac{1}{2p}-\epsilon,p}(\mathbb{R}^d))}+CT\Vert \nabla_x\Delta_x^{-1}\rho\Vert_{L^\infty([0, T], L^{p+\epsilon'}(\mathbb{R}^d))} \right)\\&\leq \Vert f_0\Vert_{H^{s,p}_{x,v}}\exp(CT+CT^{\frac{p\epsilon}{1+p\epsilon}}R+CT(1+RT^{1-\frac{1}{p+\epsilon'}}))\\&
		=\Vert f_0\Vert_{H^{s,p}_{x,v}}\exp(CT+CT^{\frac{p\epsilon}{1+p\epsilon}}R+CT^{2-\frac{1}{p+\epsilon'}}R).
	\end{flalign*}
	In the last line, we used the fact that $\rho\in K$. Now, we are ready to enter the step to prove that $G:K\to K$.
	
	First, we deal with $\Vert G[\rho]\Vert_{L^{1+p\epsilon}([0,T],H^{s+\frac{1}{2p}-\epsilon,p}(\mathbb{R}^d))}$. The problem is that $f^\rho$ has $H^{s,p}$ regularity only, while we should obtain the regularity $H^{s+\frac{1}{2p}-\epsilon,p}$ for $G[\rho]=\int_{\mathbb{R}^d}f^\rho dv$. To do this, we need the averaging lemma for $L^p$ spaces, which was accomplished by Diperna and Lions in \cite{Avg} for the whole time case. We state it here as follows.
	\begin{lemma}\cite{Avg}
		\label{Averaging}
		Let $p\in[2,\infty)$. Suppose that $h, g, k\in L^p(\mathbb{R}, L^p(\mathbb{R}^d\times\mathbb{R}^d))$, and $h$ satisfies $h(t, x, v)=0$ for $\vert v\vert \geq Q$. Suppose also that the following transport equation holds in sense of distributions:
		\begin{flalign*}
			\partial_t h+v\cdot \nabla_x h=\nabla_v\cdot g+k.
		\end{flalign*}
		Then, if we have the following estimate:
		\begin{flalign*}
			\int_\mathbb{R}\left\Vert \int_{\mathbb{R}^d} hdv\right\Vert^p_{W^{\frac{1}{2p}, p}}dt \leq C(1+Q^{d(p-1)}) \left( \Vert h \Vert^p_{L^p(\mathbb{R}\times \mathbb{R}^d\times\mathbb{R}^d)}+\Vert g \Vert^p_{L^p(\mathbb{R}\times \mathbb{R}^d\times\mathbb{R}^d)}+\Vert k\Vert^p_{L^p(\mathbb{R}\times\mathbb{R}^d\times\mathbb{R}^d)} \right).
		\end{flalign*}
	\end{lemma}
	However, this theorem itself is insufficient since for our case, the time domain is $[0, T]$, not $\mathbb{R}$. Therefore, we need a localized version of the averaging Lemma stated as follows.
	\begin{lemma}
		\label{Average_modified}
		Suppose that the following holds for $t\in[0, T]$, where $h$ satisfies $h(t, x, v)=0$ for $\vert v\vert \geq Q$:
		\begin{flalign*}
			\partial_t f+v\cdot \nabla_x h=\nabla_v\cdot g.
		\end{flalign*}
		Then, for $(\alpha-1)p>-1$, we have the following estimate, where $C$ may depend on $\alpha$ and $p$ but not on $T, Q$.
		\begin{flalign*}
			&\int_{[0, T]} t^{p\alpha} (T-t)^{p\alpha} \left\Vert \int_{\mathbb{R}^d} h dv\right\Vert^p_{W^{\frac{1}{2p}, p}}dt \\&\leq C(1+Q^{d(p-1)}) \left( (T^{2\alpha p+1}+T^{2\alpha p-p+1})\Vert h \Vert^p_{L^\infty([0, T], L^p(\mathbb{R}^d\times\mathbb{R}^d) )}+T^{2\alpha p}\Vert g \Vert^p_{L^p([0, T]\times \mathbb{R}^d\times\mathbb{R}^d)}\right).
		\end{flalign*}
	\end{lemma}
	\begin{proof}
		Suppose that the following holds for $t\in[0, T]$:
		\begin{flalign*}
			\partial_t h+v\cdot \nabla_x h=\nabla_v\cdot g.
		\end{flalign*}
		By multiplying the function $\zeta_\alpha(t):=\max\{0, t^\alpha(T-t)^\alpha\}$, we see that the following holds for whole $t$ in sense of distributions:
		\begin{flalign*}
			\partial_t [\zeta_\alpha(t)h]+v\cdot\nabla_x(\zeta_\alpha(t)f)=\nabla_v\cdot(g\zeta_\alpha(t))+h\zeta_\alpha'(t).
		\end{flalign*}
		Using the Lemma \ref{Averaging} shows that
		\begin{flalign*}
			&\int_{[0, T]}\left\Vert \int_{\mathbb{R}^d} h\zeta_\alpha(t)dv\right\Vert^p_{W^{\frac{1}{2p}, p}}dt \\&\leq C(1+Q^{d(p-1)}) \left( \Vert h\zeta_\alpha(t) \Vert^p_{L^p([0, T]\times \mathbb{R}^d\times\mathbb{R}^d)}+\Vert g\zeta_\alpha(t) \Vert^p_{L^p([0, T]\times \mathbb{R}^d\times\mathbb{R}^d)}+\Vert h\zeta_\alpha'(t)\Vert^p_{L^p([0, T]\times\mathbb{R}^d\times\mathbb{R}^d)} \right)\\&
			\leq C(1+Q^{d(p-1)}) \left( (\Vert\zeta_\alpha\Vert^p_{L^p([0, T])}+\Vert\zeta'_\alpha\Vert^p_{L^p([0, T])})\Vert h \Vert^p_{L^\infty([0, T], L^p(\mathbb{R}^d\times\mathbb{R}^d) )}+\Vert\zeta_\alpha\Vert^p_{L^\infty([0, T])}\Vert g \Vert^p_{L^p([0, T]\times \mathbb{R}^d\times\mathbb{R}^d)}\right)\\&
			\leq C(1+Q^{d(p-1)}) \left( (T^{2\alpha p+1}+T^{2\alpha p-p+1})\Vert h \Vert^p_{L^\infty([0, T], L^p(\mathbb{R}^d\times\mathbb{R}^d) )}+T^{2\alpha p}\Vert g \Vert^p_{L^p([0, T]\times \mathbb{R}^d\times\mathbb{R}^d)}\right).
		\end{flalign*}
	\end{proof}
	
	Finally, the factor $t^\alpha(T-t)^\alpha$ doesn't look nice, so we try to get rid of it using interpolation:
	\begin{lemma}
		\label{Average_interpolate}
		Under the same setting as in Lemma \ref{Average_modified}, the following inequality holds:
		\begin{flalign*}
			&\int_{0}^T \left\Vert \int_{\mathbb{R}^d} h(t, \cdot, v)dv \right\Vert^{1+p\epsilon}_{W^{\frac{1}{2p}-\frac{\epsilon}{2}, p}} dt \\&\leq C T^{(p\epsilon)^2} (1+Q^\frac{d(p-1)(1+p\epsilon)}{p}) \left((1+T^{1-p\epsilon}) \Vert h\Vert_{L^\infty([0, T], L^p(\mathbb{R}^d\times\mathbb{R}^d))}+T^{1-\frac{1}{p}} \Vert g \Vert_{L^p([0, T]\times \mathbb{R}^d\times\mathbb{R}^d)}\right)^{1+p\epsilon}
		\end{flalign*}
	\end{lemma}
	\begin{proof}
		First, denote $\rho_h:=\int_{\mathbb{R}^d} h(\cdot, \cdot, v)dv$ and notice the following inequality
		\begin{flalign*}
			\Vert \rho_h\Vert^p_{L^p(\mathbb{R}^d)}\leq C Q^{d(p-1)} \Vert h\Vert^p_{L^p(\mathbb{R}^d\times\mathbb{R}^d)}
		\end{flalign*}
		and the interpolation inequality
		\begin{flalign*}
			\Vert \rho_h\Vert_{W^{\frac{1}{2p}-\frac{\epsilon}{2} ,p}}\leq \Vert \rho_h \Vert^{p\epsilon}_{L^p}\Vert \rho_h \Vert_{W^{\frac{1}{2p},p}}^{1-p\epsilon}.
		\end{flalign*}
		These imply the following:
		\begin{flalign*}
			&\int_{0}^T \Vert \rho_h(t)\Vert^{1+p\epsilon}_{W^{\frac{1}{2p}-\frac{\epsilon}{2}, p}} dt\leq CQ^{d\epsilon(p-1)(1+p\epsilon)}\Vert h\Vert^{p\epsilon(1+p\epsilon)}_{L^\infty([0, T], L^p)}\int_0^T \Vert \rho_h(t)\Vert^{1-(p\epsilon)^2}_{W^{\frac{1}{2p},p}}dt\\&\leq CQ^{d\epsilon(p-1)(1+p\epsilon)}\Vert h\Vert^{p\epsilon(1+p\epsilon)}_{L^\infty([0, T], L^p)} \left( \int_{0}^T[(T-t)^\alpha t^\alpha \Vert \rho_h(t)\Vert_{W^{\frac{1}{2p}, p}}]^pdt \right)^\frac{1-(p\epsilon)^2}{p}\\&\times\left(\int_0^T (t^\alpha(T-t)^\alpha)^{-\frac{1-(p\epsilon)^2}{1-\frac{1-(p\epsilon)^2}{p}}}  dt \right)^{1-\frac{1-(p\epsilon)^2}{p}}\\&
			\leq  CQ^{d\epsilon(p-1)(1+p\epsilon)}\Vert h\Vert^{p\epsilon(1+p\epsilon)}_{L^\infty([0, T], L^p)} T^{-2\alpha(1-(p\epsilon)^2)+1-\frac{1-(p\epsilon)^2}{p}} (1+Q^\frac{d(p-1)(1-(p\epsilon)^2)}{p})\\&\times T^{2\alpha(1-(p\epsilon)^2)}\left((T^\frac{1}{p}+T^{-1+\frac{1}{p}}) \Vert h\Vert_{L^\infty([0, T], L^p(\mathbb{R}^d\times\mathbb{R}^d))}+ \Vert g \Vert_{L^p([0, T]\times \mathbb{R}^d\times\mathbb{R}^d)}\right)^{1-(p\epsilon)^2}\\&
			\leq C\Vert h\Vert^{p\epsilon(1+p\epsilon)}_{L^\infty([0, T], L^p)} T^{(p\epsilon)^2} (1+Q^\frac{d(p-1)(1+p\epsilon)}{p})\\&\times \left((1+T) \Vert h\Vert_{L^\infty([0, T], L^p(\mathbb{R}^d\times\mathbb{R}^d))}+T^{1-\frac{1}{p}} \Vert g \Vert_{L^p([0, T]\times \mathbb{R}^d\times\mathbb{R}^d)}\right)^{1-(p\epsilon)^2}\\
			&\leq C T^{(p\epsilon)^2} (1+Q^\frac{d(p-1)(1+p\epsilon)}{p}) \left((1+T^{1-p\epsilon}) \Vert h\Vert_{L^\infty([0, T], L^p(\mathbb{R}^d\times\mathbb{R}^d))}+T^{1-\frac{1}{p}} \Vert g \Vert_{L^p([0, T]\times \mathbb{R}^d\times\mathbb{R}^d)}\right)^{1+p\epsilon}
		\end{flalign*}
		if we pick $1-\frac{1}{p}<\alpha<\frac{1}{1-(p\epsilon)^2}-\frac{1}{p}$.
	\end{proof}
	
	We apply the above Lemma to $h=\Lambda_{x}^s f^\rho(x,v)$. It satisfies the following equation:
	\begin{flalign*}
		\partial_t(\Lambda_x^s f^\rho)+v\cdot\nabla_x(\Lambda_x^s f^\rho)=\nabla_v\cdot(\Lambda_x^s(\nabla_x\Delta_x^{-1}\rho f^\rho)).
	\end{flalign*}
	We then obtain the following estimates:
	\begin{flalign*}
		&\int_0^T\left\Vert \int_{\mathbb{R}^d}  f^\rho(x,v)dv\right\Vert^{1+p\epsilon}_{H^{s+\frac{1}{2p}-\epsilon, p}}dt=\int_0^T\left\Vert \int_{\mathbb{R}^d} \Lambda_{x}^s f^\rho(x,v)dv\right\Vert^{1+p\epsilon}_{H^{\frac{1}{2p}-\epsilon, p}}dt\leq C \int_0^T\left\Vert \int_{\mathbb{R}^d} \Lambda_{x}^s f^\rho(x,v)dv\right\Vert^{1+p\epsilon}_{W^{\frac{1}{2p}-\frac{\epsilon}{2}, p}}dt\\
		& \leq C T^{(p\epsilon)^2} (1+Q^\frac{d(p-1)(1+p\epsilon)}{p}) \\&\times\left((1+T^{1-p\epsilon}) \Vert \Lambda_{x}^s f^\rho(x,v)\Vert_{L^\infty([0, T], L^p(\mathbb{R}^d\times\mathbb{R}^d))}+T^{1-\frac{1}{p}} \Vert \Lambda_x^s(\nabla_x\Delta_x^{-1}\rho f^\rho) \Vert_{L^p([0, T]\times \mathbb{R}^d\times\mathbb{R}^d)}\right)^{1+p\epsilon}\\
		&\leq  C T^{(p\epsilon)^2} (1+Q^\frac{d(p-1)(1+p\epsilon)}{p}) \\&\times\left((1+T^{1-p\epsilon}) \Vert f_0\Vert_{H^{s,p}_{x,v}}\exp(CT+CT^{\frac{p\epsilon}{1+p\epsilon}}R+CT^{2-\frac{1}{p+\epsilon'}}R)+T^{1-\frac{1}{p}} \Vert \Lambda_x^s(\nabla_x\Delta_x^{-1}\rho f^\rho) \Vert_{L^p([0, T]\times \mathbb{R}^d\times\mathbb{R}^d)}\right)^{1+p\epsilon}.
	\end{flalign*}
	We need to estimate $\Vert \Lambda_x^s(\nabla_x\Delta_x^{-1}\rho f^\rho) \Vert_{L^p([0, T]\times \mathbb{R}^d\times\mathbb{R}^d)}$. To do this, we use the following Kato-Ponce estimates:
	\begin{lemma}\cite{KP2},\cite{KP1}
		\label{Kato Ponce}
		For $s>0$ and $1<r<\infty$, $1<p_1, p_2, q_1, q_2\leq\infty$ such that $\frac{1}{r}=\frac{1}{p_1}+\frac{1}{q_1}=\frac{1}{p_2}+\frac{1}{q_2}$, we have
		\begin{flalign*}
			\Vert \Lambda^s( h g) \Vert_{L^r(\mathbb{R}^n)}\lesssim \Vert h \Vert_{L^{p_1}(\mathbb{R}^n)}\Vert\Lambda^{s}  g\Vert_{L^{q_1}(\mathbb{R}^n)}+\Vert \Lambda^s h \Vert_{L^{p_2}(\mathbb{R}^n)}\Vert g\Vert_{L^{q_2}(\mathbb{R}^n)}. 
		\end{flalign*}
	\end{lemma}
	We apply this to obtain
	\begin{flalign*}
		\Vert \Lambda_x^s(\nabla_x\Delta_x^{-1}\rho\cdot f^\rho)(t, \cdot, v) \Vert_{L^p_x}\leq C(\Vert \Lambda_x^{s}\nabla_x\Delta_x^{-1}\rho(t) \Vert_{L^q_x}\Vert f^\rho(t, \cdot, v) \Vert_{L^r_x}+\Vert \nabla_x\Delta_x^{-1}\rho(t) \Vert_{L^\infty_x}\Vert \Lambda_x^s f^\rho(t,\cdot,  v) \Vert_{L^p_x})
	\end{flalign*}
	We choose $q$ and $r$ such that $\frac{1}{p}=\frac{1}{q}+\frac{1}{r}$ such that the following embedding holds:
	\begin{flalign*}
		H^{s',p+\epsilon'}(\mathbb{R}^d)\hookrightarrow H^{s, q}(\mathbb{R}^d)\qquad \text{and}\qquad H^{s,p}(\mathbb{R}^d)\hookrightarrow L^r(\mathbb{R}^d)
	\end{flalign*}
	This is satisfied if $\frac{1}{2p}-\epsilon-\frac{d}{p}>-\frac{d}{q}$ and $s-\frac{d}{p}>-\frac{d}{r}$. This is equivalent to $\frac{d}{p}-\frac{1}{2p}+\epsilon<\frac{d}{q}<s$ and this is possible since $s>\frac{d}{p}-\frac{1}{2p}$. Using this, we see that
	\begin{flalign*}
		\Vert \Lambda_x^s(\nabla_x\Delta_x^{-1}\rho \cdot f^\rho)(t, \cdot, v) \Vert_{L^p_x}\leq C\Vert \nabla_x\Delta_x^{-1}\rho(t) \Vert_{H^{s',p+\epsilon'}_x}\Vert \Lambda_x^s f^\rho(t, \cdot, v) \Vert_{L^p_x}
	\end{flalign*}
	Taking $p$-th power on both sides and then integrating in $v$ and time gives\footnote{We silently replaced $\Vert \Lambda_x^s f^\rho(t) \Vert_{L^p_{x,v}}$ by $\Vert \Lambda_x^{s,v} f^\rho(t) \Vert_{L^p_{x,v}}$. This is justified by Mikhlin multiplier theorem applied to Fourier multiplier operator $\Lambda_x^s\Lambda_{x,v}^{-s}$.}
	\begin{flalign*}
		&\int_0^T\Vert \Lambda_x^s(\nabla_x\Delta_x^{-1}\rho\cdot f^\rho)(t, \cdot, v) \Vert_{L^p_x}^p dt\\&\leq C\Vert f_0\Vert_{H^{s,p}_{x,v}}^p\exp(pCT+pCT^\frac{p\epsilon}{1+p\epsilon}R+pCT^{2-\frac{1}{p+\epsilon'}}R)\Vert \nabla_x\Delta_x^{-1}\rho \Vert_{L^p\left([0, T], H^{s', p+\epsilon'}_x\right)}^p\\
		&\leq C\Vert f_0\Vert_{H^{s,p}_{x,v}}^p\exp(pCT+pCT^\frac{p\epsilon}{1+p\epsilon}R+pCT^{2-\frac{1}{p+\epsilon'}}R)R^p
	\end{flalign*}
	
	Putting altogether, we obtain
	\begin{flalign*}
		&\int_0^T\Vert G[\rho](t) \Vert_{H^{s+\frac{1}{2p}-\epsilon,p}_x}^{1+p\epsilon} dt=\int_0^T\left\Vert \int_{\mathbb{R}^d}  f^\rho(x,v)dv\right\Vert^{1+p\epsilon}_{H^{s+\frac{1}{2p}-\epsilon, p}}dt\\&  \leq  C T^{(p\epsilon)^2} (1+Q^\frac{d(p-1)(1+p\epsilon)}{p}) \\&\times\left((1+T^{1-p\epsilon}) \Vert f_0\Vert_{H^{s,p}_{x,v}}\exp(CT+CT^{\frac{p\epsilon}{1+p\epsilon}}R+CT^{2-\frac{1}{p+\epsilon'}}R)+T^{1-\frac{1}{p}} \Vert \Lambda_x^s(\nabla_x\Delta_x^{-1}\rho f^\rho) \Vert_{L^p([0, T]\times \mathbb{R}^d\times\mathbb{R}^d)}\right)^{1+p\epsilon}\\
		&\leq  C T^{(p\epsilon)^2} (1+Q^\frac{d(p-1)(1+p\epsilon)}{p}) \\&\times\left((1+T^{1-p\epsilon}) \Vert f_0\Vert_{H^{s,p}_{x,v}}\exp(CT+CT^{\frac{p\epsilon}{1+p\epsilon}}R+CT^{2-\frac{1}{p+\epsilon'}}R)(1+T^{1-\frac{1}{p}}\Vert \nabla_x\Delta_x^{-1}\rho \Vert_{L^p\left([0, T], H^{s', p+\epsilon'}_x\right)})\right)^{1+p\epsilon}\\
		&\leq C T^{(p\epsilon)^2} (1+Q^\frac{d(p-1)(1+p\epsilon)}{p}) \\&\times\left((1+T^{1-p\epsilon}) \Vert f_0\Vert_{H^{s,p}_{x,v}}\exp(CT+CT^{\frac{p\epsilon}{1+p\epsilon}}R+CT^{2-\frac{1}{p+\epsilon'}}R)\right)^{1+p\epsilon}\left(1+T^{1-\frac{1}{p}}R\right)^{1+p\epsilon}.
	\end{flalign*}
	This can be made smaller than $(R/3)^p$ if we choose $T>0$ small enough.
	
	We also need to estimate
	\begin{flalign*}
		\int_0^T\Vert \nabla_x\Delta_x^{-1}G[\rho](t) \Vert^{p}_{H^{s', p+\epsilon'}}dt.
	\end{flalign*}
	This is obtained as follows. First, note that
	\begin{flalign*}
		&\left\Vert \nabla_x\Delta_x^{-1}\left( \int f^\rho(t, \cdot, v)dv \right) \right\Vert_{L^{p+\epsilon'}_x}\leq C\left\Vert  \int f^\rho(t, \cdot, v)dv  \right\Vert_{L^\frac{(p+\epsilon')d}{p+d+\epsilon'}_x} \\&\leq C \left\Vert  \int f^\rho(t, \cdot, v)dv  \right\Vert_{L^1_x}^\frac{\frac{1}{p}+\frac{1}{p'}}{1-\frac{1}{p}}\left\Vert  \int f^\rho(t, \cdot, v)dv  \right\Vert_{L^p_x}^\frac{1-\frac{1}{p'}}{1-\frac{1}{p}}\\
		&\leq \sup_{t\in[0, T]} \Vert f^\rho(t)\Vert_{L^1_{x,v}}^\frac{\frac{1}{p}+\frac{1}{p'}}{1-\frac{1}{p}} \left( Q^{d-\frac{d}{p}} \sup_{t\in[0, T]}\Vert f^\rho(t) \Vert_{L^p_{x,v}} \right)^\frac{1-\frac{1}{p'}}{1-\frac{1}{p}} \\&
		\leq  \Vert f_0\Vert_{L^1_{x,v}}^\frac{\frac{1}{p}+\frac{1}{p'}}{1-\frac{1}{p}} \left( Q^{d-\frac{d}{p}} \Vert f_0 \Vert_{L^{p}_{x,v}} \right)^\frac{1-\frac{1}{p'}}{1-\frac{1}{p}},
	\end{flalign*}
	Since we have
	\begin{flalign*}
		&\left\Vert \left(\nabla_x\Delta_x^{-1}\left( \int f^\rho(t, \cdot, v)dv \right)\right) \right\Vert_{H^{s', p+\epsilon'}_x} \leq C\left( \left\Vert \nabla_x\Delta_x^{-1}\left( \int f^\rho(t, \cdot, v)dv \right) \right\Vert_{L^{p+\epsilon'}_x}+\left\Vert \left( \int f^\rho(t, \cdot, v)dv \right) \right\Vert_{H^{s'-1, p+\epsilon'}_x} \right)\\
		&\leq  C\left( \left\Vert \nabla_x\Delta_x^{-1}\left( \int f^\rho(t, \cdot, v)dv \right) \right\Vert_{L^{p+\epsilon'}_x}+\left\Vert  \int f^\rho(t, \cdot, v)dv  \right\Vert_{H^{s, p}_x} \right)\\
		&\leq  C\left( \left\Vert \nabla_x\Delta_x^{-1}\left( \int f^\rho(t, \cdot, v)dv \right) \right\Vert_{L^{p+\epsilon'}_x}+  \int \left\Vert f^\rho(t, \cdot, v)\right\Vert_{H^{s, p}_x}dv \right) \\
		&\leq  C\left( \left\Vert \nabla_x\Delta_x^{-1}\left( \int f^\rho(t, \cdot, v)dv \right) \right\Vert_{L^{p+\epsilon'}_x}+ Q^{\frac{d(p-1)}{p}} \Vert f^\rho(t) \Vert_{H^{s,v}_{x}L^p_v} \right)\\ 
		&\leq C\left( \left\Vert \nabla_x\Delta_x^{-1}\left( \int f^\rho(t, \cdot, v)dv \right) \right\Vert_{L^{p+\epsilon'}_x}+ Q^{\frac{d(p-1)}{p}} \Vert f^\rho(t) \Vert_{H^{s,v}_{x,v}} \right) \\
		&\leq C\left( \Vert f_0\Vert_{L^1_{x,v}}^\frac{\frac{1}{p}+\frac{1}{p'}}{1-\frac{1}{p}} \left( Q^{d-\frac{d}{p}} \Vert f_0 \Vert_{L^{p}_{x,v}} \right)^\frac{1-\frac{1}{p'}}{1-\frac{1}{p}} + Q^{\frac{d(p-1)}{p}} \Vert f_0 \Vert_{H^{s,v}_{x,v}}\exp(CT+CT^{\frac{p\epsilon}{1+p\epsilon}}R+CT^{2-\frac{1}{p+\epsilon'}}R) \right),
	\end{flalign*}
	we have
	\begin{flalign*}
		&\int_0^T\Vert \nabla_x\Delta_x^{-1}G[\rho](t) \Vert^{p}_{H^{s', p+\epsilon'}}dt\\&\leq CT\left( \Vert f_0\Vert_{L^1_{x,v}}^\frac{\frac{1}{p}+\frac{1}{p'}}{1-\frac{1}{p}} \left( Q^{d-\frac{d}{p}} \Vert f_0 \Vert_{L^{p}_{x,v}} \right)^\frac{1-\frac{1}{p'}}{1-\frac{1}{p}} + Q^{\frac{d(p-1)}{p}} \Vert f_0 \Vert_{H^{s,v}_{x,v}}\exp(CT+CT^{\frac{p\epsilon}{1+p\epsilon}}R+CT^{2-\frac{1}{p+\epsilon'}}R) \right)^p,
	\end{flalign*}
	which can be made smaller than $\left(\frac{R}{3}\right)^p$ if we make $T$ small enough. 
	We finally estiamte
	\begin{flalign*}
		\int_0^T\Vert \partial_t\nabla_x\Delta_x^{-1}G[\rho](t) \Vert^{p+\epsilon'}_{L^{p+\epsilon'}_x}dt.
	\end{flalign*}
	To do this, we first observe that
	\begin{flalign*}
		\partial_t f^\rho+v\cdot\nabla_x f^\rho-\nabla_x\Delta_x^{-1}\rho\cdot\nabla_v f=0.
	\end{flalign*}
	Integrating in $v$ gives
	\begin{flalign*}
		\partial_t G[\rho]+\int_{\vert v\vert\leq Q} v\cdot \nabla_x f^\rho dv=0.
	\end{flalign*}
	We operate $\nabla_x\Delta_x^{-1}$ and take $L^{p+\epsilon'}$ norm on both sides to obtain
	\begin{flalign*}
		\Vert \partial_t\nabla_x\Delta_x^{-1}\nabla_x G[\rho](t) \Vert_{L^{p+\epsilon'}_x}&=\left\Vert \int_{\vert v\vert \leq Q} v\cdot \nabla_x\Delta_x^{-1}\nabla_x f^\rho(t, \cdot, v) dv \right\Vert_{L^{p+\epsilon'}_x}\leq  \int_{\vert v\vert \leq Q}\vert v\vert \left\Vert  \nabla_x\Delta_x^{-1}\nabla_x f^\rho(t, \cdot, v)\right\Vert_{L^{p+\epsilon'}_x} dv \\& \leq CQ\int_{\vert v\vert \leq Q} \Vert f^\rho(t, \cdot, v)\Vert_{L^{p+\epsilon'}_x}dv\leq CQ^{1+d(1-\frac{1}{p+\epsilon'})}\Vert f^\rho\Vert_{L^{p+\epsilon'}_{x,v}}\\&\leq CQ^{1+d(1-\frac{1}{p+\epsilon'})}\Vert f^\rho\Vert_{H^{s-\frac{d\epsilon'}{p(p+\epsilon')}, p+\epsilon'}_{x,v}}\leq CQ^{1+d(1-\frac{1}{p+\epsilon'})}\Vert f^\rho\Vert_{H^{s, p}_{x,v}}\\&\leq CQ^{1+d(1-\frac{1}{p+\epsilon'})}\Vert f_0\Vert_{H^{s, p}_{x,v}}\exp(CT+CT^{\frac{\epsilon p}{1+\epsilon p}}R+CT^{2-\frac{1}{p+\epsilon'}}R)
	\end{flalign*}
	Therefore, we have
	\begin{flalign*}
		\left(\int_0^T\Vert \partial_t\nabla_x\Delta_x^{-1}\rho(t) \Vert^{p+\epsilon'}_{L^p_x}dt\right)^\frac{1}{p+\epsilon'}\leq CT^\frac{1}{p+\epsilon'}Q^{1+d(1-\frac{1}{p+\epsilon'})}\Vert f_0\Vert_{H^{s, p}_{x,v}}\exp(CT+CT^{1-\frac{1}{p}}R+CT^{2-\frac{1}{p+\epsilon'}}R)
	\end{flalign*}
	and can be made smaller than $\frac{R}{3}$ if we make $T$ small enough. 
	
	Combidning these, we see that $\Vert G[\rho] \Vert_V\leq \frac{R}{3}+\frac{R}{3}+\frac{R}{3}\leq R$. Finally, by construction, we have that $G[\rho](0)=\int_{\mathbb{R}^d} f_0(\cdot, v)dv$. Therefore, we see that $G$ maps $K$ into itself.
	
	We also note that the estimates above also serve as a proof for non emptiness of $K$. Namely, we can specify the element in the set $$\left\{\rho\in V: \rho(0, x)=\int_{\mathbb{R}^d}f_0(x, v)dv \right\}.$$	
	The element is $\rho(x, t)=\int_{\mathbb{R}^d} f_0(x-vt, v)dv$. Notice that if we extend $G$ to whole $V$ with the same formula, we see that the element is just $G[0]$. Therefore, with some minor modifications, we can apply the a priori estimates obtained above to conclude that $\rho\in V$. Therefore, $K$ is nonempty and we are now free with non-emptiness issue of $K$.
	
	\subsection{Continuity of $G$}
	In this section, we prove the continuity of $G$. To do this, we suppose that the sequence
	$\rho_n$ in $K$ converges to $\rho\in K$ in the topology of $X$. Then, the following holds:
	
	\begin{enumerate}
		\item $\rho_n \rightharpoonup \rho\quad \text{in}\quad L^{1+\epsilon p}([0, T], H^{s+\frac{1}{2p}-\epsilon,p}(\mathbb{R}^d))$,\\
		\item $\nabla_x\Delta_x^{-1}\rho_n \rightharpoonup \nabla_x\Delta_x^{-1}\rho\quad \text{in}\quad L^{p}([0, T], H^{s', p+\epsilon'}(\mathbb{R}^d))$,\\
		\item $\partial_t\nabla_x\Delta_x^{-1}\rho_n \rightharpoonup \partial_t\nabla_x\Delta_x^{-1}\rho\quad \text{in}\quad L^{p+\epsilon'}([0, T], L^{p+\epsilon'}(\mathbb{R}^d))$.
	\end{enumerate}
	The first fact implies that 
	\begin{flalign*}
		\nabla_x^2\Delta_x^{-1}\rho_n \rightharpoonup \nabla_x^2\Delta_x^{-1}\rho\quad \text{in}\quad L^{1+\epsilon p}([0, T], H^{s+\frac{1}{2p}-\epsilon,p}(\mathbb{R}^d)),
	\end{flalign*}
	which leads to
	\begin{flalign*}
		\nabla_x^2\Delta_x^{-1}\rho_n \rightharpoonup \nabla_x^2\Delta_x^{-1}\rho\quad \text{in}\quad L^{1+\epsilon p}([0, T], H^{s',p+\epsilon'}(\mathbb{R}^d)).
	\end{flalign*}
	The second fact implies that
	\begin{flalign*}
		\nabla_x\Delta_x^{-1}\rho_n(x, t)=\nabla_x\Delta_x^{-1}\int f_0(x, v)dv+\int_0^t \partial_\tau\nabla_x\Delta_x^{-1}\rho_n(x, \tau)d\tau
	\end{flalign*}
	converges weakly to
	\begin{flalign*}
		\nabla_x\Delta_x^{-1}\rho(x, t)=\nabla_x\Delta_x^{-1}\int f_0(x, v)dv+\int_0^t \partial_\tau\nabla_x\Delta_x^{-1}\rho(x, \tau)d\tau
	\end{flalign*}
	in $L^{1+\epsilon p}([0, T], L^{p+\epsilon'}(\mathbb{R}^d))$. Therefore, we see that
	\begin{flalign*}
		\nabla_x\Delta_x^{-1}\rho_n(x, t)\rightharpoonup \nabla_x\Delta_x^{-1}\rho(x, t)\quad \text{in}\quad L^{1+\epsilon p}([0, T], H^{s'+1,p+\epsilon'}(\mathbb{R}^d))
	\end{flalign*}
	
	For any compact subset $O\subset\mathbb{R}^d$, we see that the following inclusions are compact:
	\begin{flalign*}
		H^{s'+1, p+\epsilon'}(O)\hookrightarrow H^{s'+1-\delta, p+\epsilon'}(O)\hookrightarrow L^{p+\epsilon'}(O)
	\end{flalign*}
	where $\delta>0$ is sufficiently small. Then, by the Aubin-Lions Lemma, we see that
	\begin{flalign*}
		\nabla_x\Delta_x^{-1}\rho_n \to \nabla_x\Delta_x^{-1}\rho\quad \text{in}\quad L^{1+\epsilon p}([0, T], H^{s'-\delta,p+\epsilon'}(O)).
	\end{flalign*}
	Since $\delta>0$ is sufficiently small, we see that
	\begin{flalign*}
		\nabla_x\Delta_x^{-1}\rho_n \to \nabla_x\Delta_x^{-1}\rho\quad \text{in}\quad L^{1+\epsilon p}([0, T], C^1(O)).
	\end{flalign*}
	Now, we are ready to prove the (weak) convergence of $G[\rho_n]$ to $G[\rho]$.
	First, since $\Vert G[\rho_n]\Vert_V\leq R$, we see by Banach-Alaoglu theorem that there exists a convergent subsequence. We show that such limit must be $G[\rho]$. For then it follows that any convergent subsequence must have the same limit, which implies that whole sequence must converge to $G[\rho]$, and the continuity is proven.
	
	It remains to prove that the limit is $G[\rho]$. We first prove that $f^{\rho_n}$ converges to $f^{\rho}$ in sense of distributions. To do this, we first prove that the inverse flow maps $(\Phi^{\rho_n}_t)^{-1}$ converges locally uniformly to $(\Phi^\rho_t)^{-1}$. We state these facts as a lemma.
	\begin{lemma}
		The flow maps $\Phi^{\rho_n}_t$, as well as its inverse $(\Phi^{\rho_n}_t)^{-1}$ converges locally uniformly. That is, for any compact subset $O$, we have that
		\begin{flalign*}
			&\Phi^{\rho_n}_t \to \Phi^{\rho}_t\quad \text{uniformly on}\;O,\;\text{and}\\
			&(\Phi^{\rho_n}_t)^{-1} \to (\Phi^{\rho}_t)^{-1}\quad \text{uniformly on}\;O.
		\end{flalign*}
	\end{lemma}
	\begin{proof}
		We first show that for any compact subset $O$, $(x, v)\in O$ cannot escape another compact subset $O'$ during $t\in[0,T]$ along any flow $\Phi^{\rho_n}$ and $\Phi^{\rho}$. For this, suppose that $O\subset \{\vert x\vert \leq B, \vert v\vert \leq B\}$. Then, We first have $$\vert V^{\rho_n}(t, x, v) \vert\leq \vert v\vert+CRT^{\frac{\epsilon p}{1+\epsilon p}}+CT(1+T^{1-\frac{1}{p+\epsilon'}}R)$$ regardless of $x$ and consequently,
		\begin{flalign*}
			\vert X^{\rho_n}(t, x, v) \vert \leq \vert x \vert +\int_{0}^T \vert V^{\rho_n}(t, x,v )\vert dt \leq \vert x\vert +T\vert v\vert +CRT^\frac{1+2\epsilon p}{1+\epsilon p}+CT^2(1+T^{1-\frac{1}{p+\epsilon'}}R).
		\end{flalign*}
		Thus, if $(x,v)\in O$, then we have
		\begin{flalign*}
			\vert X^{\rho_n}(t, x, v) \vert \leq (1+T)B+CRT^{\frac{1+2\epsilon p}{1+\epsilon p}}+CT^2(1+T^{1-\frac{1}{p+\epsilon'}}R)
		\end{flalign*}
		and
		\begin{flalign*}
			\vert V^{\rho_n}(t, x, v) \vert\leq B+CRT^{\frac{\epsilon p}{1+\epsilon p}}+CT(1+T^{1-\frac{1}{p+\epsilon'}}R),
		\end{flalign*}
		which shows that the flow cannot escape
		\begin{flalign*}
			O'=\{ \vert x\vert \leq (1+T)B+CRT^{\frac{1+2\epsilon p}{1+\epsilon p}}+CT^2(1+T^{1-\frac{1}{p+\epsilon'}}R), \vert v\vert \leq B+CRT^{\frac{\epsilon p}{1+\epsilon p}}+CT(1+T^{1-\frac{1}{p+\epsilon'}}R) \}.
		\end{flalign*}
		Same argument holds also for $\Phi^\rho$.
		
		Now, observe that for $(x,v)\in O$,
		\begin{flalign*}
			&\vert\Phi^{\rho_n}(t, x, v)-\Phi^\rho(t,x,v)\vert\\&\leq \int_0^t \left\vert \left(V^{\rho_n}(s,x,v)-V^\rho(s,x,v), \nabla_x\Delta_x^{-1}\rho_n(s,\Phi^{\rho_n}(s,x,v))-\nabla_x\Delta_x^{-1}\rho(s,\Phi^{\rho}(s,x,v)) \right) \right\vert ds\\&\leq \int_0^t\left\vert V^{\rho_n}(s,x,v)-V^\rho(s,x,v) \right\vert ds+\int_0^t \left\vert \nabla_x\Delta_x^{-1}\rho_n(s,\Phi^{\rho_n}(s,x,v))-\nabla_x\Delta_x^{-1}\rho_n(s,\Phi^{\rho}(s,x,v)) \right\vert ds\\&+\int_0^t\left\vert \nabla_x\Delta_x^{-1}\rho_n(s,\Phi^\rho(s,x,v))-\nabla_x\Delta_x^{-1}\rho(s,\Phi^{\rho}(s,x,v)) \right\vert ds \\&\leq \int_0^t\left\vert V^{\rho_n}(s,x,v)-V^\rho(s,x,v) \right\vert ds+\int_0^t  \Vert \nabla_x\Delta_x^{-1}\rho_n(s) \Vert_{C^1(O')}\vert \Phi^{\rho_n}(s,x,v)-\Phi^\rho(s,x,v) \vert  ds\\&+\int_0^t\left\Vert \nabla_x\Delta_x^{-1}\rho_n(s)-\nabla_x\Delta_x^{-1}\rho(s) \right\Vert_{L^\infty(O')} ds 
		\end{flalign*}
		Therefore, by Gronwall's inequality, we have
		\begin{flalign*}
			\vert\Phi^{\rho_n}(t, x, v)-\Phi^\rho(t,x,v)\vert&\leq \int_0^T\left\Vert \nabla_x\Delta_x^{-1}\rho_n(s)-\nabla_x\Delta_x^{-1}\rho(s) \right\Vert_{L^\infty(O')} ds \\&\exp\left(\int_0^T (1+\Vert \nabla_x\Delta_x^{-1}\rho_n(s) \Vert_{C^1(O')}) ds\right)
		\end{flalign*}
		The term in the exponential is bounded by $(1+CR)T$ while the term $$\int_0^T\left\Vert \nabla_x\Delta_x^{-1}\rho_n(s)-\nabla_x\Delta_x^{-1}\rho(s) \right\Vert_{L^\infty(O')} ds$$ converges to zero since 
		\begin{flalign*}
			\nabla_x\Delta_x^{-1}\rho_n \to \nabla_x\Delta_x^{-1}\rho\quad \text{in}\quad L^{1+\epsilon p}([0, T], C^1(O')).
		\end{flalign*}
		Therefore, the flow map converges uniformly on $O$ for all $t\in[0,T]$.
		
		For the uniform convergence of the inverse flow map, we first need the Lipschitz bound for the flow map. This is again obtained via Gronwall's inequality. We have
		\begin{flalign*}
			\nabla_{x,v}\Phi^{\rho_n}(t, x, v)=I+\int_0^t [\nabla_{x,v}(v, -\nabla_x\Delta_x^{-1}\rho_n)\circ\Phi^{\rho_n}(s,x,v)]\nabla_{x,v}\Phi^{\rho_n}(s, x, v) ds,
		\end{flalign*}
		which implies
		\begin{flalign*}
			\Vert \nabla_{x,v}\Phi^{\rho_n}(t) \Vert_{L^\infty(\mathbb{R}^d)} \leq C+\int_0^t (1+\Vert\nabla^2\Delta_x^{-1}\rho_n(s)\Vert_{L^\infty(\mathbb{R}^d)} )\Vert \nabla_{x,v}\Phi^{\rho_n}(s) \Vert_{L^\infty(\mathbb{R}^d)}ds.
		\end{flalign*}
		Applying the Gronwall inequality gives
		\begin{flalign*}
			\Vert \nabla_{x,v}\Phi^{\rho_n}(t) \Vert_{L^\infty(\mathbb{R}^d)} \leq C\exp(T+\Vert\nabla^2\Delta_x^{-1}\rho_n\Vert_{L^1([0, T], L^\infty(\mathbb{R}^d))})\leq C\exp(T+CRT^{\frac{\epsilon p}{1+\epsilon p}})
		\end{flalign*}
		Now, note that by the formula of the inverse matrix and the fact that our flows have unit Jacobian determinant, we see that $\Vert \nabla_{x,v}(\Phi^{\rho_n}(t))^{-1} \Vert_{L^\infty(\mathbb{R}^d)}$ are also uniformly bounded as well.
		
		Finally, we prove the uniform convergence of the inverse flow map. To show this, we see that for any compact subset $O$ and for $(x,v)\in O$ that
		\begin{flalign*}
			\vert (\Phi_t^{\rho_n})^{-1}(x, v)-(\Phi_t^\rho)^{-1}(x, v) \vert &\leq \Vert \nabla_{x,v}(\Phi^{\rho_n}(t))^{-1} \Vert_{L^\infty(\mathbb{R}^d)}\vert \Phi^{\rho_n}(t,(\Phi_t^{\rho_n})^{-1}(x, v))-\Phi^{\rho_n}(t, (\Phi_t^\rho)^{-1}(x, v)) \vert \\&\leq \Vert\nabla_{x,v}(\Phi^{\rho_n}(t))^{-1} \Vert_{L^\infty(\mathbb{R}^d)}\vert \Phi^{\rho}(t,(\Phi_t^{\rho})^{-1}(x, v))-\Phi^{\rho_n}(t, (\Phi_t^\rho)^{-1}(x, v)) \vert
		\end{flalign*}
		The first term is uniformly bounded, while the second term converges to zero as $n\to\infty$ because $(\Phi_t^{\rho})^{-1}(x, v)$ is trapped compact subset.\footnote{To see this rigorously, define the map $\Psi(t, x, v)=(\Phi_t^{\rho})^{-1}(x, v)$. This map is continuous jointly in variables. Now, we see that $(\Phi_t^{\rho})^{-1}(x, v)\in \Psi([0, T]\times O)$, which is compact.} The proof of the Lemma is complete.
	\end{proof}
	Now, we prove that $f^{\rho_n}$ converges to $f^{\rho}$ in sense of distributions. We first take $g\in C^\infty_c(\mathbb{R}^d)$ such that $\Vert f_0-g \Vert_{L^p(\mathbb{R}^d)}<\epsilon$. Then, we have for any compact subset $O$ that
	\begin{flalign*}
		&\Vert f^{\rho_n}(t)-f^\rho(t) \Vert_{L^p(O)}=\Vert f_0\circ(\Phi_t^{\rho_n})^{-1}-f_0\circ(\Phi_t^{\rho})^{-1} \Vert_{L^p(O)}\\&\leq \Vert f_0\circ(\Phi_t^{\rho_n})^{-1}-g\circ(\Phi_t^{\rho_n})^{-1} \Vert_{L^p(O)}+\Vert g\circ(\Phi_t^{\rho_n})^{-1}-g\circ(\Phi_t^{\rho})^{-1} \Vert_{L^p(O)}+\Vert g\circ(\Phi_t^{\rho_n})^{-1}-f_0\circ(\Phi_t^{\rho})^{-1} \Vert_{L^p(O)}\\&\leq 2\epsilon+\Vert g\circ(\Phi_t^{\rho_n})^{-1}-g\circ(\Phi_t^{\rho})^{-1} \Vert_{L^p(O)}
	\end{flalign*}
	We have used that our flow is measure preserving in the last line. By the fact that $g$ is smooth and compactly supported and that $(\Phi^{\rho_n}_t)^{-1} \to (\Phi^{\rho}_t)^{-1}$ uniformly on $O$, we see that the last term converges to zero as $n\to\infty$. Since $\epsilon>0$ was arbitrary, we see that $$f^{\rho_n}\to f^\rho\quad\text{on}\quad  L^p_{loc}(\mathbb{R}^d)$$ and in particular, converges in sense of distributions.
	
	Thus, since $f^{\rho_n}$ and $f$ are compactly supported in $v$, we see that $G[\rho_n]=\int_\mathbb{R}^d f^{\rho_n}(t, x, v) dv$ converges to $G[\rho]=\int_\mathbb{R}^d f^{\rho}(t, x, v) dv$ in sense of distributions, which means that the limit is the same for all subsequences. Therefore, we have shown the desired result, and the proof for continuity of $G$ is complete.
	\subsection{Existence of solution}
	We can now say that $G$ has a fixed point by the Schauder-Tychonoff fixed point theorem. Let $\rho$ be the fixed point. We can see that $f^\rho$ is the solution for the Vlasov-Poisson equation. This completes the proof of the existence.
	\section{Uniqueness and Continuity of Solutions}\label{unique}
	In this section, we prove the uniqueness and continuity of solutions. 
	\subsection{Uniqueness}
	To show uniqueness, we first need to show that every solution of the equation is transported along the flow. The first step is to show the well definedness of the flow of the vector field $\left(v, -\nabla_x\Delta_x^{-1}\int f(t, x, v)dv \right)$. Note that direct application of a priori estimates is impossible since we have assumed the regularity of $\rho$ and then showed the regularity of $f$. In this situation, we have to show the regularity of $\rho$ by assuming the regularity of $f$.
	First, we have 
	Suppose that we have a solution $f\in L^\infty([0, T], H^{s, p}\cap L^1(\mathbb{R}^d\times\mathbb{R}^d))$ of \eqref{eq:VP} such that $f(t, x, v)=0$ for $\vert v\vert \geq Q$. We have to obtain the regularity of $$\nabla_x\Delta_x^{-1}\left( \int f(t, \cdot, v)dv \right).$$
	
	First, we have
	\begin{flalign*}
		&\left\Vert \nabla_x\Delta_x^{-1}\left( \int f(t, \cdot, v)dv \right) \right\Vert_{L^{p+\epsilon'}_x}\leq C\left\Vert  \int f(t, \cdot, v)dv  \right\Vert_{L^\frac{(p+\epsilon')d}{p+d+\epsilon'}_x} \\&\leq C \left\Vert  \int f(t, \cdot, v)dv  \right\Vert_{L^1_x}^\frac{\frac{1}{p}+\frac{1}{p'}}{1-\frac{1}{p}}\left\Vert  \int f(t, \cdot, v)dv  \right\Vert_{L^p_x}^\frac{1-\frac{1}{p'}}{1-\frac{1}{p}}\\
		&\leq \sup_{t\in[0, T]} \Vert f(t)\Vert_{L^1_{x,v}}^\frac{\frac{1}{p}+\frac{1}{p'}}{1-\frac{1}{p}} \left( Q^{d-\frac{d}{p}} \sup_{t\in[0, T]}\Vert f(t) \Vert_{L^p_{x,v}} \right)^\frac{1-\frac{1}{p'}}{1-\frac{1}{p}} \\&
		\leq \sup_{t\in[0, T]} \Vert f(t)\Vert_{L^1_{x,v}}^\frac{\frac{1}{p}+\frac{1}{p'}}{1-\frac{1}{p}} \left( Q^{d-\frac{d}{p}} \sup_{t\in[0, T]}\Vert f(t) \Vert_{H^{s,p}_{x,v}} \right)^\frac{1-\frac{1}{p'}}{1-\frac{1}{p}},
	\end{flalign*}
	which is uniformly bounded in $t$. Here again, $p'=\frac{(p+\epsilon')d}{p+\epsilon'+d}$. 
	
	We also have by Lemma \ref{Averaging}, Lemma \ref{Kato Ponce}, and the Sobolev embeddings we have used in the existence proof, that
	\begin{flalign*}
		& \int_0^T\left\Vert \int_{\mathbb{R}^d}  f(x,v)dv\right\Vert^{1+p\epsilon}_{H^{s', p+\epsilon'}(\mathbb{R}^d)}dt \leq C\int_0^T\left\Vert \int_{\mathbb{R}^d}  f(x,v)dv\right\Vert^{1+p\epsilon}_{H^{s+\frac{1}{2p}-\epsilon, p}(\mathbb{R}^d)}dt\\&\leq C\int_0^T\left\Vert \int_{\mathbb{R}^d} \Lambda_{x}^s fdv\right\Vert^{1+p\epsilon }_{H^{\frac{1}{2p}-\epsilon, p}(\mathbb{R}^d)}dt\leq C \int_0^T\left\Vert \int_{\mathbb{R}^d} \Lambda_{x}^s f(x,v)dv\right\Vert^{1+p\epsilon}_{W^{\frac{1}{2p}, p}(\mathbb{R}^d)}dt\\
		&\leq C T^{(p\epsilon)^2} (1+Q^\frac{d(p-1)(1+p\epsilon)}{p}) \\&\times\left((1+T^{1-p\epsilon}) \Vert f(x,v)\Vert_{L^\infty([0, T], H^{s,p}(\mathbb{R}^d\times\mathbb{R}^d))}+T^{1-\frac{1}{p}} \Vert \Lambda_x^s(\nabla_x\Delta_x^{-1}\rho f) \Vert_{L^p([0, T]\times \mathbb{R}^d\times\mathbb{R}^d)}\right)^{1+p\epsilon}\\
		&\leq C T^{(p\epsilon)^2} (1+Q^\frac{d(p-1)(1+p\epsilon)}{p}) \\&\times\left((1+T^{1-p\epsilon}) \Vert f(x,v)\Vert_{L^\infty([0, T], H^{s,p}(\mathbb{R}^d\times\mathbb{R}^d))}\right)^{1+p\epsilon}\left( 1+T^{1-\frac{1}{p}}\Vert \nabla_x\Delta_x^{-1}\rho \Vert_{L^p\left([0, T], H^{s', p+\epsilon'}_x\right)} \right)^{1+p\epsilon}\\
		&\leq C T^{(p\epsilon)^2} (1+Q^\frac{d(p-1)(1+p\epsilon)}{p}) \\&\times\left((1+T^{1-p\epsilon}) \Vert f(x,v)\Vert_{L^\infty([0, T], H^{s,p}(\mathbb{R}^d\times\mathbb{R}^d))}\right)^{1+p\epsilon}\left( 1+T\Vert \nabla_x\Delta_x^{-1}\rho \Vert_{L^\infty\left([0, T], H^{s', p+\epsilon'}_x\right)} \right)^{1+p\epsilon}
	\end{flalign*}
	holds. To bound the last term, we simply use
	\begin{flalign*}
		&\left\Vert \left(\nabla_x\Delta_x^{-1}\left( \int f(t, \cdot, v)dv \right)\right) \right\Vert_{H^{s', p+\epsilon'}_x} \leq C\left( \left\Vert \nabla_x\Delta_x^{-1}\left( \int f(t, \cdot, v)dv \right) \right\Vert_{L^{p+\epsilon'}_x}+\left\Vert \left( \int f(t, \cdot, v)dv \right) \right\Vert_{H^{s'-1, p+\epsilon'}_x} \right)\\
		&\leq  C\left( \left\Vert \nabla_x\Delta_x^{-1}\left( \int f(t, \cdot, v)dv \right) \right\Vert_{L^{p+\epsilon'}_x}+\left\Vert  \int f(t, \cdot, v)dv  \right\Vert_{H^{s, p}_x} \right)\\
		&\leq  C\left( \left\Vert \nabla_x\Delta_x^{-1}\left( \int f(t, \cdot, v)dv \right) \right\Vert_{L^{p+\epsilon'}_x}+  \int \left\Vert f(t, \cdot, v)\right\Vert_{H^{s, p}_x}dv \right) \\
		&\leq  C\left( \left\Vert \nabla_x\Delta_x^{-1}\left( \int f(t, \cdot, v)dv \right) \right\Vert_{L^{p+\epsilon'}_x}+ Q^{\frac{d(p-1)}{p}} \Vert f(t) \Vert_{H^{s,v}_{x}L^p_v} \right)\\ 
		&\leq C\left( \left\Vert \nabla_x\Delta_x^{-1}\left( \int f(t, \cdot, v)dv \right) \right\Vert_{L^{p+\epsilon'}_x}+ Q^{\frac{d(p-1)}{p}} \Vert f(t) \Vert_{H^{s,v}_{x,v}} \right) \\
		&\leq C\left(\sup_{t\in[0, T]} \Vert f(t)\Vert_{L^1_{x,v}}^\frac{\frac{1}{p}+\frac{1}{p'}}{1-\frac{1}{p}} \left( Q^{d-\frac{d}{p}} \sup_{t\in[0, T]}\Vert f(t) \Vert_{H^{s,p}_{x,v}} \right)^\frac{1-\frac{1}{p'}}{1-\frac{1}{p}} + Q^{\frac{d(p-1)}{p}} \Vert f(t) \Vert_{H^{s,v}_{x,v}} \right)
	\end{flalign*}
	whose right hand side is uniformly bounded by assumption on $f$. Here, we used the fact that $f(t, x, v)=0$ for $\vert v\vert\geq Q$ in the fourth line. Combining these results, we see that the vector field $\nabla_x\Delta_x^{-1}\left( \int f(t, \cdot, v)dv \right)$ is in the class $L^p([0, T], H^{s'+1, p+\epsilon'}(\mathbb{R}^d))$ and in particular, in the class $L^1([0, T], C^1(\mathbb{R}^d))$. Therefore, we see that the flow map generated by $\left(v, -\nabla_x\Delta_x^{-1}\int f(t, x, v)dv \right)$ is well defined.
	
	Now, we show that the solution is actually advected by the flow. We observe that two functions $f$ and $f_0\circ\Phi_t^{-1}$ are both $L^p$ (weak) solutions to the following transport equation:
	\begin{flalign*}
		\partial_t h(t, x, v)+v\cdot \nabla_x(t, x, v) h-\nabla_x\Delta_x^{-1}\int_{\mathbb{R}^d}f(t, x, v)dv\cdot \nabla_v h(t, x, v)=0.
	\end{flalign*}
	By the uniqueness of $L^p$ solutions to the transport equation with Lipschitz velocity, we see that $f=f_0\circ\Phi_t^{-1}$, which show that the claim is true.
	
	We now prove uniqueness.\footnote{Note that the above argument does not show the uniqueness, since the velocity field itself depends on the solution, which can be different. All we have shown above is that the existing solutions are advected by their own flow.} It follows the proof of Theorem 2.1 in the paper \cite{Loeper}, with minor modifications.
	Suppose that there are two solutions, $f_1$ and $f_2$, with initial data $f_0$. For simplicity, we assume from now on that $f_{0} \ge 0$. It is clarified in \cite[Section 4]{Loeper} that the non-negativity assumption is not necessary at all. Let $$\Phi_i(t, x, v)=(X_i(t, x, v),V_i(t, x, v))$$ be the corresponding characteristic flow for $i=1,2$. We aim to estimate the following quantity:
	\begin{flalign*}
		P(t):=\frac{1}{2}\iint  f_0(x, v) \vert \Phi_1(t, x, v)-\Phi_2(t, x, v) \vert^2 dvdx.
	\end{flalign*}
	Differentiating in time, we have
	\begin{flalign*}
		\frac{d}{dt}P(t)&=\iint f_0(x, v)(\Phi_1(t, x, v)-\Phi_2(t, x, v))\cdot(\partial_t \Phi_1(t, x, v)-\partial_t\Phi_2(t, x, v))\\&
		=\iint f_0(x, v) (X_1(t, x, v)-X_2(t, x, v))\cdot(V_1(t, x, v)-V_2(t, x, v))dvdx \\&-\iint f_0(x, v)(V_1(t, x, v)-V_2(t, x, v))\cdot(\nabla_x \Delta_x^{-1}\rho_1(t, X_1(t, x, v))-\nabla_x\Delta_x^{-1}\rho_2(t, X_2(t, x, v))) dvdx \\& =: A_1+A_2. 
	\end{flalign*}
	By the Cauchy--Schwarz inequality, $\vert A_1\vert\leq P(t)$. On the other hand, $A_2$ is bounded by 
	\begin{flalign*}
	&\left[ \iint f_0(x, v)\vert V_1(t, x, v)-V_2(t, x, v) \vert^2 dvdx  \right]^\frac{1}{2} \\&\times \left[\iint f_0(x, v)\vert \nabla_x\Delta_x^{-1} \rho_1(t, X_1(t, x, v))-\nabla_x\Delta_x^{-1}\rho_2(t, X_2(t, x, v)) \vert^2 dvdx \right]^\frac{1}{2}\\&\leq (2P(t))^\frac{1}{2}\left[ \iint f_0(x, v)\vert \nabla_x\Delta_x^{-1} \rho_1(t, X_1(t, x, v))-\nabla_x\Delta_x^{-1}\rho_2(t, X_2(t, x, v)) \vert^2 dvdx \right]^\frac{1}{2}\\&
		\leq (2P(t))^\frac{1}{2}\left[ \iint f_0(x, v)\vert \nabla_x\Delta_x^{-1} \rho_1(t, X_1(t, x, v))-\nabla_x\Delta_x^{-1}\rho_2(t, X_1(t, x, v)) \vert^2 dvdx \right]^\frac{1}{2}\\& \quad + (2P(t))^\frac{1}{2}\left[ \iint f_0(x, v)\vert \nabla_x\Delta_x^{-1} \rho_2(t, X_1(t, x, v))-\nabla_x\Delta_x^{-1}\rho_2(t, X_2(t, x, v)) \vert^2 dvdx \right]^\frac{1}{2} \\ & =: (2P(t))^\frac{1}{2}\left([S_1(t)]^\frac{1}{2}+[S_2(t)]^\frac{1}{2}\right). 
	\end{flalign*}
	From \cite[Theorem 2.9]{Loeper}, we can estimate $S_1(t)$ as 
	\begin{flalign*}
		S_1(t)\leq 2\max\{\Vert \rho_1(t) \Vert_{L^\infty}, \Vert\rho_2(t)\Vert_{L^\infty}\}^2 P(t)\leq 4(\Vert \rho_1(t) \Vert^2_{H^{s+\frac{1}{2p}-\epsilon,p}}+\Vert \rho_2(t) \Vert^2_{H^{s+\frac{1}{2p}-\epsilon,p}})P(t). 
	\end{flalign*} 
	The estimate of $S_2(t)$ needs to be done differently from the case of \cite{Loeper}: we estimate 
	\begin{flalign*}
		S_2(t)&\leq \iint f_0(x, v) \Vert \nabla_x^2\Delta_x^{-1} \rho_2(t) \Vert_{C(\mathbb{R}^d)}^2\vert X_1(t, x, v)-X_2(t, x, v) \vert^2dvdx\\&
		\leq \Vert\rho_2(t)\Vert_{H^{s+\frac{1}{2p}-\epsilon, p}(\mathbb{R}^d)}^2 \iint f_0(x, v)\vert X_1(t, x, v)-X_2(t, x, v) \vert^2dvdx 
		\leq C\Vert \rho_2(t) \Vert^2_{H^{s+\frac{1}{2p}-\epsilon, p}(\mathbb{R}^d)}P(t). 
	\end{flalign*}
	Combining these estimates, we have the following inequality:
	\begin{flalign*}
		\frac{d}{dt}P(t)\leq CP(t)\left(1+\Vert \rho_1(t) \Vert^2_{H^{s+\frac{1}{2p}-\epsilon, p}}+\Vert \rho_2(t) \Vert^2_{H^{s+\frac{1}{2p}-\epsilon, p}}\right)^\frac{1}{2}. 
	\end{flalign*}
	The Gronwall's inequality gives
	\begin{flalign*}
		P(t)&\leq P(0)\exp\left( CT+C\int_0^T\left( \Vert\rho_1(t')\Vert_{H^{s+\frac{1}{2p}-\epsilon, p}}+\Vert\rho_2(t')\Vert_{H^{s+\frac{1}{2p}-\epsilon, p}} \right)dt' \right)\\&\leq P(0)\exp\left(CT+CT^{\frac{\epsilon p}{1+\epsilon p}}\left(\Vert\rho_1\Vert_{L^{1+\epsilon p}([0, T];\,H^{s+\frac{1}{2p}-\epsilon, p})}+\Vert\rho_2\Vert_{L^{1+\epsilon p}([0, T];\,H^{s+\frac{1}{2p}-\epsilon, p})}\right)\right).
	\end{flalign*}
	Since $P(0)=0$, we conclude that $P(t)=0$. This means that $\Phi_1(x, v, t)=\Phi_2(x, v, t)$ for all points $(x, v)$ where $f_0(x, v)\neq 0$. Since $f_1$ and $f_2$ are transported along the flow, as we have seen earlier, $f_1=f_2$, and the proof is complete. \qedsymbol 
	\subsection{Continuity in time}
	So far, we have obtained the existence and uniqueness of solutions in $L^\infty([0, T], L^1\cap H^{s,p}(\mathbb{R}^d\times\mathbb{R}^d))$. In this section, we show that such solutions actually belongs to $C([0, T], L^1\cap H^{s,p}(\mathbb{R}^d\times\mathbb{R}^d))$. Continuity with respect to $L^1$ follows from the fact that $f$ is transported along the flow. However, continuity in $H^{s,p}$ does not follow from this, since the norm is highly nonlocal.
	
	Our strategy for proving continuity is standard; using the Lipschitz continuity in time to show weak continuity in time and using the a proiri estimates of $f$ to show continuity in the norm. Let us start with the first one.
	
	As we stated earlier, we need Lipschitz continuity to do so. Let us recall that
	\begin{flalign*}
		\partial_t f=-v\cdot \nabla_x f+\nabla_x\Delta_x^{-1}\int_{\mathbb{R}^d} f(\cdot, v)dv\cdot \nabla_v f=0.
	\end{flalign*}
	Taking $H^{-1,p}_{x,v}$ norm on both sides, we obtain the following:
	\begin{flalign*}
		\Vert \partial_t f \Vert_{H^{-1,p}_{x,v}} &\leq \Vert -v\cdot \nabla_x f \Vert_{H^{-1,p}_{x,v}}+\left\Vert \nabla_x\Delta_x^{-1}\int_{\mathbb{R}^d} f(\cdot, v)dv\cdot \nabla_v f \right\Vert_{H^{-1,p}_{x,v}}\\&\leq C  \Vert -v\cdot \nabla_x f \Vert_{H^{-1,p}_{x}L^p_v}+ C\left\Vert \nabla_x\Delta_x^{-1}\int_{\mathbb{R}^d} f(\cdot, v)dv\cdot \nabla_v f \right\Vert_{H^{-1,p}_v L^p_x}\\&\leq CQ\Vert \nabla_x f \Vert_{H^{-1,p}_x L^p_v}+C\left\Vert \nabla_x\Delta_x^{-1}\int_{\mathbb{R}^d} f(\cdot, v)dv \right\Vert_{L^\infty_x} \Vert \nabla_v f \Vert_{H^{-1,p}_v L^p_x}\\
		&\leq CQ\Vert f \Vert_{L^p_{x,v}} +C\left\Vert \nabla_x\Delta_x^{-1}\int_{\mathbb{R}^d} f(\cdot, v)dv \right\Vert_{H^{s+\frac{1}{2p}-\epsilon}_x}\Vert f \Vert_{L^p_{x,v}}\\
		&\leq CQ\Vert f \Vert_{L^p_{x,v}} +C\left\Vert \int_{\vert v\vert\leq Q} f(\cdot, v)dv \right\Vert_{H^{s}_x}\Vert f \Vert_{L^p_{x,v}}\\
		&\leq CQ\Vert f \Vert_{L^p_{x,v}} +C\int_{\vert v\vert\leq Q}\left\Vert  f(\cdot, v) \right\Vert_{H^{s}_x}dv \cdot \Vert f \Vert_{L^p_{x,v}}\\
		&\leq C\Vert f\Vert_{L^p_{x,v}}\left( Q+CQ^{1-\frac{1}{p}}\Vert f\Vert_{H^{s}_xL^p_v}  \right)\\
		&\leq C\Vert f\Vert_{L^p_{x,v}}\left( Q+CQ^{1-\frac{1}{p}}\Vert f\Vert_{H^{s}_{x,v}}  \right)=M.
	\end{flalign*}
	
	Now, we show the weak continuity in time, as follows. Pick a test function $l\in H^{-s, \frac{p}{p-1}}(\mathbb{R}^d\times\mathbb{R}^d)$. Since the space of smooth test function is dense in $H^{-s, \frac{p}{p-1}}(\mathbb{R}^d\times\mathbb{R}^d)$, we can pick a smooth test function $\varphi$ such that $\Vert l-\varphi\Vert_{H^{-s, \frac{p}{p-1}}(\mathbb{R}^d\times\mathbb{R}^d)}<\frac{\epsilon}{4\sup_{\tau\in[0, T]}\Vert f(\tau) \Vert_{H^{s, p}_{x,v}}}$. Now, we have that
	\begin{flalign*}
		&\vert \langle f(t), l \rangle-\langle f(t'), l \rangle  \vert \leq \vert \langle f(t), l \rangle-\langle f(t), \varphi \rangle  \vert+\vert \langle f(t), \varphi \rangle-\langle f(t'), \varphi \rangle  \vert+\vert \langle f(t'), \varphi \rangle-\langle f(t'), l \rangle  \vert\\
		&\leq \Vert f(t) \Vert_{H^{s, p}_{x,v}} \Vert l-\varphi \Vert_{H^{-s, \frac{p}{p-1}}_{x,v}}+\Vert f(t)-f(t') \Vert_{H^{s, p}_{x,v}} \Vert \varphi \Vert_{H^{-s, \frac{p}{p-1}}_{x,v}} +\Vert f(t') \Vert_{H^{s, p}_{x,v}} \Vert l-\varphi \Vert_{H^{-s, \frac{p}{p-1}}_{x,v}}\\
		&\leq 2\sup_{\tau\in[0, T]} \Vert f(\tau)\Vert_{H^{s, p}_{x,v}} \cdot \frac{\epsilon}{4\sup_{\tau\in[0, T]}\Vert f(\tau) \Vert_{H^{s, p}_{x,v}}} + \int_{t}^{t'} \Vert \partial_t f(\tau) \Vert_{H^{-1, p}_{x,v}}  d\tau \cdot\Vert \varphi \Vert_{H^{1, \frac{p}{p-1}}_{x,v}} \\& <\frac{\epsilon}{2}+ \Vert \varphi \Vert_{H^{1, \frac{p}{p-1}}_{x,v}} \vert t-t'\vert M, 
	\end{flalign*}
	which can be made smaller than $\epsilon$ if have $$\vert t-t'\vert <\frac{\epsilon}{2\Vert \varphi \Vert_{H^{1, \frac{p}{p-1}}_{x,v}}M}.$$
	We now have the weak continuity in time. This also implies
	\begin{flalign*}
		\Vert f(t)\Vert_{H^{s, p}_{x,v}}  \leq \liminf_{t'\to t} \Vert f(t')\Vert_{H^{s, p}_{x,v}}
	\end{flalign*}
	and hence we obtain the lower semicontinuity in the norm.
	
	We turn to the proof of upper semicontinuity. Pick time $t\in[0, T]$. Assume that $t'>t$, since the opposite cases can be done in the same way. Then, we have a solution with initial data $f(t)$, and by uniqueness, this should coincide with $f(\cdot-t)$. We now have a similar a priori estimates for $f(\cdot-t)$:
	\begin{flalign*}
		\Vert f(t')\Vert_{H^{s,p}_{x,v}} \leq \Vert f(t)\Vert_{H^{s,p}_{x,v}}\exp(C(t'-t)+C(t'-t)^{1-\frac{1}{p}}R+C(t'-t)^{2-\frac{1}{p+\epsilon'}}R)
	\end{flalign*}
	and therefore, we have
	\begin{flalign*}
		\limsup_{t'\to t+0} \Vert f(t')\Vert_{H^{s,p}_{x,v}} \leq \limsup_{t'\to t+0} \Vert f(t)\Vert_{H^{s,p}_{x,v}}\exp(C(t'-t)+C(t'-t)^{1-\frac{1}{p}}R+C(t'-t)^{2-\frac{1}{p+\epsilon'}}R) = \Vert f(t)\Vert_{H^{s,p}_{x,v}}.
	\end{flalign*}
	Therefore, we have that
	\begin{flalign*}
		\lim_{t'\to t} \Vert f(t')\Vert_{H^{s,p}_{x,v}} = \Vert f(t)\Vert_{H^{s,p}_{x,v}}.
	\end{flalign*}
	By the weak continuity and continuity in the norm, we have strong continuity, and the proof of the continuity in time is complete.
	\section{Modifications of the Proof for the Case $d=1$}
	\label{1D}
	In this section, we show how to modify the proof in previous sections to prove the case $d=1$. The main difficulty is that the electric field $-\nabla_x\Delta_x^{-1}\rho$ is not in $L^p$ for any finite $p$.
	The first modification is the choice of $V$ and $X$. We choose them as follows:
	\begin{flalign*}
		V=\{\rho:[0, T]\times \mathbb{R}\to \mathbb{R}\,\vert\, &\rho\in L^{1+\epsilon p}([0, T], H^{s+\frac{1}{2p}-\epsilon, p}(\mathbb{R})),\; \nabla_x\Delta_x^{-1}\rho\in L^p([0, T], \dot{H}^{s+\frac{1}{2p}-\epsilon,p}(\mathbb{R})), \\& \;\text{and}\;  \partial_t \nabla_x\Delta_x^{-1}\rho\in L^{p}([0, T], L^p(\mathbb{R}))\}.
	\end{flalign*}
	where $\epsilon>0$ is again sufficiently small. The choice of $K$ is identical with the previous proof.
	The second modification of the proof is that as we mentioned before, we have that $\nabla_x\Delta_x^{-1}\rho$ is not in $L^p$. However, it is in $L^\infty$ uniformly in time by the following argument.
	
	We first have that
	\begin{flalign*}
		\Vert \nabla_x\Delta_x^{-1}\rho(0)\Vert_{L^\infty(\mathbb{R})}\leq \frac{1}{2} \Vert \rho(0)\Vert_{L^1(\mathbb{R})}\leq \frac{1}{2}\Vert f_0\Vert_{L^1(\mathbb{R}^2)}
	\end{flalign*}
	since we have
	\begin{flalign*}
		\vert\nabla_x\Delta_x^{-1}\rho(0, x)\vert=\frac{1}{2}\left\vert \int_{\mathbb{R}}\sgn(x-y)\rho(0, y)dy \right\vert.
	\end{flalign*}
	Also, we have
	\begin{flalign*}
		\Vert \partial_t\nabla_x\Delta_x^{-1}\rho\Vert_{L^1([0, T], L^\infty(\mathbb{R}))}&\leq C \Vert \partial_t\nabla_x\Delta_x^{-1}\rho\Vert_{L^1([0, T], H^{\frac{1}{p}+\epsilon, p}(\mathbb{R}))}\leq CT^{1-\frac{1}{p}}\Vert \partial_t\nabla_x\Delta_x^{-1}\rho\Vert_{L^p([0, T], H^{\frac{1}{p}+\epsilon, p}(\mathbb{R}))}\\&\leq CRT^{1-\frac{1}{p}}.
	\end{flalign*}
	Combining these results gives that
	\begin{flalign*}
		\Vert \nabla_x\Delta_x^{-1}\rho(t)\Vert_{L^\infty(\mathbb{R})}\leq \frac{1}{2}\Vert f_0\Vert_{L^1(\mathbb{R}^2)}+CRT^{1-\frac{1}{p}}
	\end{flalign*}
	for any $t\in[0, T]$, which gives
	\begin{flalign*}
		\Vert \nabla_x\Delta_x^{-1}\rho\Vert_{L^\infty([0, T], L^\infty(\mathbb{R}))}\leq \frac{1}{2}\Vert f_0\Vert_{L^1(\mathbb{R}^2)}+CRT^{1-\frac{1}{p}}.
	\end{flalign*}
	The third modification is the A priori estimates used to estimate $\Vert f^\rho\Vert_{H^{s,p}(\mathbb{R}^2)}$. It seems that we are in trouble since $\nabla_x\Delta_x^{-1}\rho$ is not in $L^p$, which implies that it is not in $H^{s+\frac{1}{2p}-\epsilon, p}$. However, if we look into the proof of the paraproduct estimates in section \S \ref{Para}, we see that we can slightly modify the proof to show that the following estimate holds:
	\begin{flalign*}
		&\frac{\partial}{\partial t}\Vert  f^\rho\Vert_{H^{s,p}(\mathbb{R}^2)} \leq C\left(1+\Vert \rho \Vert_{H^{s+\frac{1}{2p}-\epsilon,p}(\mathbb{R})}+ \Vert \nabla_x\Delta_x^{-1}\rho \Vert_{L^{\infty}(\mathbb{R})}\right)\Vert f^\rho \Vert_{H^{s,p}(\mathbb{R}^2)}.
	\end{flalign*}
	We therefore can use the Gronwall inequality to show that
	\begin{flalign*}
		\Vert f^\rho(t)\Vert_{H^{s,p}_{x,v}}&\leq \Vert f_0\Vert_{H^{s,p}_{x,v}}\exp\left(CT+C\Vert \rho \Vert_{L^1([0, T], H^{s+\frac{1}{2p}-\epsilon, p}(\mathbb{R}))} +C\Vert \nabla_x\Delta_x^{-1}\rho \Vert_{L^1([0, T], L^\infty(\mathbb{R}))} \right) \\&\leq \Vert f_0\Vert_{H^{s,p}_{x,v}}\exp\left(CT +CRT^\frac{\epsilon p}{1+\epsilon p}+CT\left(\frac{1}{2}\Vert f_0\Vert_{L^1(\mathbb{R}^2)}+CRT^{1-\frac{1}{p}}\right) \right).
	\end{flalign*}
	The fourth modification is the moment we try to estimate $\Vert G[\rho] \Vert_{L^p([0, T], H^{s+\frac{1}{2p}, p}(\mathbb{R}))}$. There is no problem with proceeding as in the previous argument until we face to estimate the following quantity:
	\begin{flalign*}
		\Vert \Lambda_x^s(\nabla_x\Delta_x^{-1}\rho f^\rho) \Vert_{L^p([0, T]\times \mathbb{R}^2)}.
	\end{flalign*}
	We still have to use the Kato Ponce inequality, but since $\nabla_x\Delta_x^{-1}\rho$ is not in $L^p$, the standard form does not work. However, it is also known that the following homogeneous version of Kato Ponce inequality also works\cite{KP2}, \cite{KP1}:
	\begin{flalign*}
		\Vert D^s( h g) \Vert_{L^r(\mathbb{R}^n)}\lesssim \Vert h \Vert_{L^{p_1}(\mathbb{R}^n)}\Vert D^{s}  g\Vert_{L^{q_1}(\mathbb{R}^n)}+\Vert D^s h \Vert_{L^{p_2}(\mathbb{R}^n)}\Vert g\Vert_{L^{q_2}(\mathbb{R}^n)}. 
	\end{flalign*}
	Here, $D=\sqrt{-\Delta}$. Using this and the standard Holder's inequality, we see that we can actually show that
	\begin{flalign*}
		\Vert \Lambda^s( h g) \Vert_{L^r(\mathbb{R}^n)}\lesssim \Vert h \Vert_{L^{p_1}(\mathbb{R}^n)}\Vert\Lambda^{s}  g\Vert_{L^{q_1}(\mathbb{R}^n)}+\Vert D^s h \Vert_{L^{p_2}(\mathbb{R}^n)}\Vert g\Vert_{L^{q_2}(\mathbb{R}^n)}. 
	\end{flalign*}
	Using this, we see that
	\begin{flalign*}
		&\Vert \Lambda_x^s(\nabla_x\Delta_x^{-1}\rho(t, \cdot) f^\rho(t, \cdot, v)) \Vert_{L^p(\mathbb{R})}\leq C(\Vert D_x^{s}\nabla_x\Delta_x^{-1}\rho(t) \Vert_{L^q_x}\Vert f^\rho(t, \cdot, v) \Vert_{L^r_x}+\Vert \nabla_x\Delta_x^{-1}\rho(t) \Vert_{L^\infty_x}\Vert \Lambda_x^s f^\rho(t,\cdot,  v) \Vert_{L^p_x}).
	\end{flalign*}
	We choose $q$ and $r$ such that $\frac{1}{p}=\frac{1}{q}+\frac{1}{r}$ such that the following embedding holds:
	\begin{flalign*}
		\dot{H}^{s+\frac{1}{2p}-\epsilon,p}(\mathbb{R})\hookrightarrow \dot{H}^{s, q}(\mathbb{R})\qquad \text{and}\qquad H^{s,p}(\mathbb{R})\hookrightarrow L^r(\mathbb{R})
	\end{flalign*}
	This is satisfied if $\frac{1}{2p}-\epsilon-\frac{1}{p}>-\frac{1}{q}$ and $s-\frac{1}{p}>-\frac{1}{r}$. This is equivalent to $\frac{1}{2p}+\epsilon<\frac{1}{q}<s$ and this is possible since $s>\frac{1}{2p}$.
	
	Then, one has
	\begin{flalign*}
		&\Vert \Lambda_x^s(\nabla_x\Delta_x^{-1}\rho(t, \cdot) f^\rho(t, \cdot, v)) \Vert_{L^p(\mathbb{R})}\\&\leq C(\Vert \nabla_x\Delta_x^{-1}\rho(t) \Vert_{\dot{H}^{s+\frac{1}{2p}-\epsilon,p}_x}\Vert f^\rho(t, \cdot, v) \Vert_{H^{s,p}_x}+\Vert \nabla_x\Delta_x^{-1}\rho(t) \Vert_{L^\infty_x}\Vert \Lambda_x^s f^\rho(t,\cdot,  v) \Vert_{L^p_x}).
	\end{flalign*}
	Integrating in $v$ gives
	\begin{flalign*}
		&\Vert \Lambda_x^s(\nabla_x\Delta_x^{-1}\rho(t, \cdot) f^\rho(t, \cdot)) \Vert_{L^p(\mathbb{R}^2)}\\
		&\leq C  \Vert  f^\rho(t,\cdot) \Vert_{H^{s,p}_{x,v}}(\Vert \nabla_x\Delta_x^{-1}\rho(t) \Vert_{\dot{H}^{s+\frac{1}{2p}-\epsilon,p}_x}+\Vert \nabla_x\Delta_x^{-1}\rho(t) \Vert_{L^\infty_x})\\
		&\leq C \Vert f_0\Vert_{H^{s,p}_{x,v}}\exp\left( CT +CRT^{\frac{\epsilon p}{1+\epsilon p}}+CT\left(\frac{1}{2}\Vert f_0\Vert_{L^1(\mathbb{R}^2)}+CRT^{1-\frac{1}{p}}\right) \right)\\&\times(\Vert \nabla_x\Delta_x^{-1}\rho(t) \Vert_{\dot{H}^{s+\frac{1}{2p}-\epsilon,p}_x}+\frac{1}{2}\Vert f_0\Vert_{L^1(\mathbb{R}^2)}+CRT^{1-\frac{1}{p}})
	\end{flalign*}
	and hence, we obtain
	\begin{flalign*}
		&\Vert \Lambda_x^s(\nabla_x\Delta_x^{-1}\rho(t, \cdot) f^\rho(t, \cdot)) \Vert_{L^p([0, T]\times\mathbb{R}^2)}\\&\leq  C \Vert f_0\Vert_{H^{s,p}_{x,v}}\exp\left( CT +CRT^{\frac{\epsilon p}{1+\epsilon p}}+CT\left(\frac{1}{2}\Vert f_0\Vert_{L^1(\mathbb{R}^2)}+CRT^{1-\frac{1}{p}}\right) \right)\\&\times\left(\left(\int_0^T \Vert \nabla_x\Delta_x^{-1}\rho(t) \Vert^p_{\dot{H}^{s+\frac{1}{2p}-\epsilon,p}_x} dt\right)^\frac{1}{p}+\frac{1}{2}T^{\frac{1}{p}}\Vert f_0\Vert_{L^1(\mathbb{R}^2)}+CRT\right)\\&=  C \Vert f_0\Vert_{H^{s,p}_{x,v}}\exp\left( CT +CRT^{\frac{\epsilon p}{1+\epsilon p}}+CT\left(\frac{1}{2}\Vert f_0\Vert_{L^1(\mathbb{R}^2)}+CRT^{1-\frac{1}{p}}\right) \right)\left(R+\frac{1}{2}T^{\frac{1}{p}}\Vert f_0\Vert_{L^1(\mathbb{R}^2)}+CRT\right)\\&\leq  C \Vert f_0\Vert_{H^{s,p}_{x,v}}\exp\left( CT +CRT^{\frac{\epsilon p}{1+\epsilon p}}+CT\left(\frac{1}{2}\Vert f_0\Vert_{L^1(\mathbb{R}^2)}+CRT^{1-\frac{1}{p}}\right) \right)\left(R+\frac{1}{2}T^{\frac{1}{p}}\Vert f_0\Vert_{L^1(\mathbb{R}^2)}+CRT\right)
	\end{flalign*}
	Using this, we see that
	\begin{flalign*}
		&\int_0^T\left\Vert \int_{\mathbb{R}}  f^\rho(x,v)dv\right\Vert^{1+\epsilon p}_{H^{s+\frac{1}{2p}-\epsilon, p}}dt \\&
		\leq C T^{(p\epsilon)^2} (1+Q^\frac{d(p-1)(1+p\epsilon)}{p}) \\&\times\left((1+T^{1-p\epsilon}) \Vert \Lambda_{x}^s f^\rho(x,v)\Vert_{L^\infty([0, T], L^p(\mathbb{R}^2))}+T^{1-\frac{1}{p}} \Vert \Lambda_x^s(\nabla_x\Delta_x^{-1}\rho f^\rho) \Vert_{L^p([0, T]\times \mathbb{R}^2)}\right)^{1+p\epsilon}\\&
		\leq C T^{(p\epsilon)^2} (1+Q^\frac{d(p-1)(1+p\epsilon)}{p})\Vert f_0\Vert^{1+\epsilon p}_{H^{s, p}_{x,v}}\exp\left( CT +CRT^{\frac{\epsilon p}{1+\epsilon p}}+CT\left(\frac{1}{2}\Vert f_0\Vert_{L^1(\mathbb{R}^2)}+CRT^{1-\frac{1}{p}}\right) \right) \\&\times\left((1+T^{1-p\epsilon})\left(1+RT^{1-\frac{1}{p}}+\frac{1}{2}T\Vert f_0\Vert_{L^1(\mathbb{R}^2)}+CRT^{2-\frac{1}{p}}\right)\right)^{1+p\epsilon}
	\end{flalign*}
	which can be made smaller than $\left( \frac{R}{3} \right)^p$ if we make $T$ small enough.
	
	To bound the quantity
	\begin{flalign*}
		\int_0^T\Vert \nabla_x\Delta_x^{-1}G[\rho](t) \Vert^{p}_{\dot{H}^{s+\frac{1}{2p}-\epsilon, p}_x}dt,
	\end{flalign*} 
	we proceed as
	\begin{flalign*}
		&\Vert \nabla_x\Delta_x^{-1}G[\rho](t) \Vert_{\dot{H}^{s+\frac{1}{2p}-\epsilon, p}_x}=\left\Vert \nabla_x\Delta_x^{-1}\int_{\vert v\vert\leq Q} f^\rho(t, \cdot, v)dv \right\Vert_{\dot{H}^{s+\frac{1}{2p}-\epsilon, p}_x}\\&=\left\Vert \int_{\vert v\vert\leq Q} f^\rho(t, \cdot, v)dv \right\Vert_{\dot{H}^{s+\frac{1}{2p}-\epsilon-1, p}_x}\leq \int_{\vert v\vert\leq Q} \left\Vert f^\rho(t, \cdot, v)\right\Vert_{\dot{H}^{s+\frac{1}{2p}-\epsilon-1, p}_x}dv 
	\end{flalign*}
	We consider two cases:
	\begin{enumerate}
		\item $s+\frac{1}{2p}-1-\epsilon\geq0$:\\
		In this case, we have
		\begin{flalign*}
			&\int_{\vert v\vert\leq Q} \left\Vert f^\rho(t, \cdot, v)\right\Vert_{\dot{H}^{s+\frac{1}{2p}-\epsilon-1, p}_x}dv \leq CQ^{1-\frac{1}{p}}\Vert f^\rho(t)\Vert_{H^{s, p}_xL^p_v}\leq CQ^{1-\frac{1}{p}}\Vert f^\rho\Vert_{H^{s, p}_{x,v}}\\&\leq CQ^{1-\frac{1}{p}}\Vert f_0\Vert_{H^{s,p}_{x,v}}\exp\left(CT +CRT^\frac{\epsilon p}{1+\epsilon p}+CT\left(\frac{1}{2}\Vert f_0\Vert_{L^1(\mathbb{R}^2)}+CRT^{1-\frac{1}{p}}\right) \right)
		\end{flalign*}
		and thus
		\begin{flalign*}
			\int_0^T\Vert \nabla_x\Delta_x^{-1}G[\rho](t) \Vert^{p}_{\dot{H}^{s+\frac{1}{2p}-\epsilon, p}_x}dt \leq CTQ^{p-1}\Vert f_0\Vert_{H^{s,p}_{x,v}}^p\exp\left(CT +CRT^\frac{\epsilon p}{1+\epsilon p}+CT\left(\frac{1}{2}\Vert f_0\Vert_{L^1(\mathbb{R}^2)}+CRT^{1-\frac{1}{p}}\right) \right)
		\end{flalign*}
		\item $s+\frac{1}{2p}-1-\epsilon<0$:\\
		In this case, we have
		\begin{flalign*}
			&\int_{\vert v\vert\leq Q} \left\Vert f^\rho(t, \cdot, v)\right\Vert_{\dot{H}^{s+\frac{1}{2p}-\epsilon-1, p}_x}dv \leq C\int_{\vert v\vert\leq Q} \left\Vert f^\rho(t, \cdot, v)\right\Vert_{L^q_x}dv\leq C\int_{\vert v\vert\leq Q} \left\Vert f^\rho(t, \cdot, v)\right\Vert_{L^p_x}^{\beta}\left\Vert f^\rho(t, \cdot, v)\right\Vert_{L^1_x}^{1-\beta}dv\\
			&\leq C\int_{\vert v\vert\leq Q} \left\Vert f^\rho(t, \cdot, v)\right\Vert_{L^p_x}^{\beta}\left\Vert f^\rho(t, \cdot, v)\right\Vert_{L^1_x}^{1-\beta}dv\leq C\int_{\vert v\vert\leq Q} \left(\beta\left\Vert f^\rho(t, \cdot, v)\right\Vert_{L^p_x}+(1-\beta)\left\Vert f^\rho(t, \cdot, v)\right\Vert_{L^1_x}\right)dv\\
			&\leq C\left( \beta Q^{1-\frac{1}{p}}\Vert f_0 \Vert_{L^p_{x,v}}  +(1-\beta)\Vert f_0\Vert_{L^1_{x,v}} \right)
		\end{flalign*}
		where $\frac{1}{q}=\frac{1}{p}-(s+\frac{1}{2p}-1-\epsilon)$ and $\frac{1}{q}=\frac{\beta}{p}+(1-\beta)$ and therefore
		\begin{flalign*}
			\int_0^T\Vert \nabla_x\Delta_x^{-1}G[\rho](t) \Vert^{p}_{\dot{H}^{s+\frac{1}{2p}-\epsilon, p}_x}dt \leq CT\left( \beta Q^{1-\frac{1}{p}}\Vert f_0 \Vert_{L^p_{x,v}}  +(1-\beta)\Vert f_0\Vert_{L^1_{x,v}} \right)^p.
		\end{flalign*}
	\end{enumerate}
	In either case, by choosing $T>0$ sufficiently small, we can bound the quantity by $\left(\frac{R}{3}\right)^p$.
	
	Finally, bounding the quantity
	\begin{flalign*}
		\int_0^T\Vert \partial_t\nabla_x\Delta_x^{-1}G[\rho](t) \Vert^{p}_{L^{p}_x}dt.
	\end{flalign*}
	can be done similarly as in the case $d\geq2$.
	
	The final major modification is the step proving the continuity of $G$. We still want to use Aubin-Lions lemma, but $\nabla_x\Delta_x^{-1}\rho$ is not in $L^p$ globally. However, since it is in $L^\infty$, our $\nabla_x\Delta_x^{-1}\rho$ is in $L^p$ locally, so it makes the use of Aubin-Lions lemma valid, since we are jsut using it on compact subsets anyway.
	
	Uniqueness and continuity of solution can be done in almost the same way.
	\section{Proof of A Priori Estimates Using Paraproduct Estimates}
	\label{Para}
	In this section we prove the a priori estimates for $f$ proposed in Section \ref{Existence}. Before starting, let us define some notations. In the $(x, v)$-space, which we need when controlling $f^\rho$, we use the following Littlewood-paley decomposition. Let $\phi_j(\cdot)=\phi(2^{-j}\cdot)$, where $\phi$ is a radial smooth function supported on $\frac{3}{4}\leq \vert (x,v)\vert\leq \frac{8}{3}$ and satisfies $\sum_{j=-\infty}^\infty \phi_j=1$ and let $\psi_{j}=\sum_{k\leq j}\phi_k$. Let $\widetilde{\phi}_j(\cdot)=\widetilde{\phi}(2^{-j}\cdot)$ where $\widetilde{\phi}$ is a radial smooth function supported in $\frac{1}{2}\leq \vert (x,v)\vert \leq 4$ such that $\widetilde{\phi}(x)=1$ on $\supp \phi$ and $\widetilde{\phi}_j(\cdot)=\widetilde{\phi}(2^{-j}\cdot)$. We denote $\phi_j(D_{x,v})$ and $\psi_j(D_{x,v})$ the Littlewood-Paley operators. In $x$-space, which we need when controlling $\nabla_x\Delta_x^{-1}\rho$, we use the same notation, However, to aviod confusion, we declare the variables. Finally, we denote $F^{s}_{p, q}=\{ f\in\mathcal{S}': \phi_j(D_x)f\in L^p(l^q) \}$ the Triebel Lizorkin space.
	
	Let us start with the proof. We recall the following:
	\begin{flalign*}
		[\Lambda_{x,v}^s, \nabla_x\Delta_x^{-1}\rho\cdot \nabla_v]f^\rho&=[\Lambda_{x,v}^s, T_{\nabla_x\Delta_x^{-1}\rho}]\cdot \nabla_vf^\rho+(T_{\nabla_x\Delta_x^{-1}\rho}\cdot \nabla_v\Lambda_{x,v}^s f^\rho-\nabla_x\Delta_x^{-1}\rho\cdot \nabla_v\Lambda_{x,v}^s f^\rho)\\
		&\;+\Lambda_{x,v}^sT_{\nabla_v f^\rho}\nabla_x\Delta_x^{-1}\rho+\Lambda_{x,v}^s R(\nabla_x\Delta_x^{-1}\rho, \nabla_v f^\rho) = \mathcal{T}_1+\mathcal{T}_2+\mathcal{T}_3+\mathcal{T}_4.
	\end{flalign*}
	where we denote $T_fg:= \sum_{k\geq 4} \psi_{k-3}(D)f \phi_k(D)g $ the paraproduct and $R(f, g)=fg-T_fg-T_gf$ be the remainder.
	We have to estimate each term separately to obtain the bound
	\begin{flalign*}
		\Vert[\Lambda_{x,v}^s, \nabla_x\Delta_x^{-1}\rho\cdot \nabla_v]f^\rho\Vert_{L^p_{x,v}}\leq C\Vert \nabla_x\Delta_x^{-1}\rho \Vert_{H^{s+1+\frac{1}{2p}-\epsilon-\frac{\epsilon'd}{p(p+\epsilon')},p+\epsilon'}_x}\Vert f \Vert_{H^{s,p}_{x,v}}.
	\end{flalign*}
	From now on, we write $\nabla_x\Delta_x^{-1}\rho=A$ and $f^\rho=f$ for simplicity. 
	\subsection{Term $\mathcal{T}_1$}
	We rewrite term $\mathcal{T}_1$ in the following form:
	\begin{flalign*}
		&\iiint e^{i(\xi+\eta, \mu)\cdot(x,v)}\left[ \langle \xi+\eta, \mu \rangle^s-\langle (\xi, \mu) \rangle \right]\sum_{k\geq 4} \psi_{k-3}(\eta)\widehat{A}(\eta)\cdot \phi_k(\xi, \mu) i\mu\widehat{f}(\xi, \mu)d\xi d\mu d\eta \\
		&=\iiint e^{i(\xi+\eta, \mu)\cdot(x,v)}\left[\int_0^1 s\langle (\xi+t\eta, \mu)\rangle^{s-2}(\xi+t\eta, \mu) dt\right]\cdot\eta\\&\times\sum_{k\geq 4} \psi_{k-3}(\eta)\widehat{A}(\eta)\cdot \phi_k(\xi, \mu) i\mu\widehat{f}(\xi, \mu)d\xi d\mu d\eta \\
		&=\iiint e^{i(\xi+\eta, \mu)\cdot(x,v)}\left[\mu\langle(\xi, \mu)\rangle^{s-1}\widehat{f}(\xi, \mu)\right]\left[i\eta\widehat{A}(\eta, \nu)\right]\cdot\\&\times\sum_{k\geq 4}\left[ \psi_{k-3}(\eta)\cdot \phi_k(\xi, \mu) \langle(\xi, \mu)\rangle^{1-s}\int_0^1 s\langle (\xi, \mu+t\nu)\rangle^{s-2}(\xi+t\eta, \mu) dt\right]d\xi d\mu d\eta 
	\end{flalign*}
	The bilinear Fourier multiplier
	\begin{flalign*}
		 \sum_{k\geq 4}\psi_{k-3}(\eta)\cdot \phi_k(\xi, \mu) \langle(\xi, \mu)\rangle^{1-s}\int_0^1 s\langle (\xi+t\eta, \mu)\rangle^{s-2}(\xi+t\eta, \mu) dt
	\end{flalign*}
	belongs to the Coifmann-Meyer class defined as in page 561 in \cite{Grafakos} as one can check, so we can apply the Coifmann-Meyer theorem to apply
	\begin{flalign*}
		\Vert \mathcal{T}_1 \Vert_{L^p_{x,v}}\leq C\Vert \nabla_{x,v}A(t, x) \Vert_{L^\infty_x} \Vert \Lambda^{s-1}_{x,v}\nabla_v f(t, x, v) \Vert_{L^p_{x,v}}\leq C \Vert A(t) \Vert_{H^{s+1+\frac{1}{2p}-\epsilon-\frac{\epsilon'd}{p(p+\epsilon')},p+\epsilon'}_x} \Vert f(t) \Vert_{H^{s,p}_{x,v}},
	\end{flalign*}
	where in the last inequality, we used the fact that $A$ is independent of $v$ and the Sobolev embedding $H^{s+1+\frac{1}{2p}-\epsilon-\frac{\epsilon'd}{p(p+\epsilon')},p+\epsilon'}\hookrightarrow C^1$.
	\subsection{Term $\mathcal{T}_2$}
	We write $\mathcal{T}_2$ as follows:
	\begin{flalign*}
		\mathcal{T}_2=\sum_{j\geq 1}(\psi_{j+2}(D_{x,v})\nabla_v\Lambda^s_{x,v}f)\cdot \phi_j(D_x)A+ (\psi_{3}(D_{x,v})\nabla_v\Lambda^s_{x,v}f)\cdot \psi_0(D_x)A
	\end{flalign*}
	We estimate the $L^p$ norm of each term in the sum. Then, this becones:
	\begin{flalign*}
		&\Vert (\psi_{j+2}(D_{x,v})\nabla_v\Lambda^s_{x,v}f)\cdot \phi_j(D_x)A \Vert_{L^p_{x,v}}\leq \Vert \psi_{j+2}(D_{x,v})\nabla_v\Lambda^s_{x,v}f \Vert_{L^p_{x,v}} \Vert\phi_j(D_x)A\Vert_{L^\infty_x}\\
		&\leq C\cdot 2^j\Vert \psi_{j+2}(D_{x,v})\Lambda^s_{x,v}f \Vert_{L^p_{x,v}} 2^\frac{jd}{p+\epsilon'} \Vert\phi_j(D_x)A\Vert_{L^{p+\epsilon'}_x}\\
		&\leq C\cdot 2^{j(1+\frac{d}{p+\epsilon'})}\Vert \Lambda^s_{x,v}f \Vert_{L^p_{x,v}} 2^{j(-s-1-\frac{1}{2p}+\epsilon+\frac{\epsilon'd}{p(p+\epsilon')})}2^{j(s+1+\frac{1}{2p}-\epsilon-\frac{\epsilon'd}{p(p+\epsilon')})} \Vert\phi_j(D_x)A\Vert_{L^{p+\epsilon'}_x}\\
		&\leq C\cdot 2^{j(-s+\frac{d}{p}-\frac{1}{2p}+\epsilon)}\Vert \Lambda^s_{x,v}f \Vert_{L^p_{x,v}} 2^{j(s+1+\frac{1}{2p}-\epsilon-\frac{\epsilon'd}{p(p+\epsilon')})} \Vert\phi_j(D_x)A\Vert_{L^{p+\epsilon'}_x}.
	\end{flalign*}
	Summing over $j\geq 1$ gives:
	\begin{flalign*}
		&\sum_{j=1}^\infty \Vert (\psi_{j+2}(D_{x,v})\nabla_v\Lambda^s_{x,v}f)\cdot \phi_j(D_x)A \Vert_{L^p_{x,v}}\leq C\sum_{j=1}^\infty 2^{j(-s+\frac{d}{p}-\frac{1}{2p}+\epsilon)}\Vert \Lambda^s_{x,v}f \Vert_{L^p_{x,v}} 2^{j(s+1+\frac{1}{2p}-\epsilon-\frac{\epsilon'd}{p(p+\epsilon')})} \Vert\phi_j(D_x)A\Vert_{L^{p+\epsilon'}_x}\\
		&\leq C \Vert \Lambda^s_{x,v}f \Vert_{L^p_{x,v}}\sum_{j=1}^\infty 2^{j(-s+\frac{d}{p}-\frac{1}{2p}+\epsilon)} 2^{j(s+1+\frac{1}{2p}-\epsilon-\frac{\epsilon'd}{p(p+\epsilon')})} \Vert\phi_j(D_x)A\Vert_{L^{p+\epsilon'}_x}\\&\leq C \Vert \Lambda^s_{x,v}f \Vert_{L^p_{x,v}}\left( \sum_{j=1}^\infty 2^{\frac{j(p+\epsilon')}{p+\epsilon'-1}(-s+\frac{d}{p}-\frac{1}{2p}+\epsilon)} \right)^{1-\frac{1}{p+\epsilon'}}\left( \sum_{j=1}^\infty 2^{j(s+1+\frac{1}{2p}-\epsilon-\frac{\epsilon'd}{p(p+\epsilon')})(p+\epsilon')} \Vert\phi_j(D_x)A\Vert_{L^{p+\epsilon'}_x}^{p+\epsilon'} \right)^\frac{1}{p+\epsilon'} \\&\leq C\Vert \Lambda^s_{x,v}f \Vert_{L^p_{x,v}} \Vert A\Vert_{\left(F^{s+1+\frac{1}{2p}-\epsilon-\frac{\epsilon'd}{p(p+\epsilon')}}_{p+\epsilon',p+\epsilon'}\right)_{x}}\\&\leq C\Vert f \Vert_{H^{s,p}_{x,v}} \Vert A\Vert_{H^{s+1+\frac{1}{2p}-\epsilon-\frac{\epsilon'd}{p(p+\epsilon')},p+\epsilon'}_{x}},
	\end{flalign*}
	where we used the fact that $s>\frac{d}{p}-\frac{1}{2p}+\epsilon$ when moving to third line to fourth line.
	We also estimate the second term. This becomes:
	\begin{flalign*}
		\Vert (\psi_{3}(D_{x,v})\nabla_v\Lambda^s_{x,v}f)\cdot \psi_0(D_x)A \Vert_{L^p_{x,v}}&\leq \Vert \psi_{3}(D_{x,v})\nabla_v\Lambda^s_{x,v}f \Vert_{L^p_{x,v}} \Vert\psi_0(D_x)A\Vert_{L^\infty_x}\leq C\Vert \Lambda^s_{x,v}f \Vert_{L^p_{x,v}} \Vert A\Vert_{L^\infty_x}\\&\leq C\Vert f \Vert_{H^{s,p}_{x,v}}\Vert A \Vert_{H^{s+1+\frac{1}{2p}-\epsilon-\frac{\epsilon'd}{p(p+\epsilon')},p+\epsilon'}_x}
	\end{flalign*}
	Therefore, we have
	\begin{flalign*}
		\Vert \mathcal{T}_2\Vert_{L^p}\leq C\Vert f \Vert_{H^{s,p}_{x,v}}\Vert A \Vert_{H^{s+1+\frac{1}{2p}-\epsilon-\frac{\epsilon'd}{p(p+\epsilon')},p+\epsilon'}_x}.
	\end{flalign*}
	\subsection{Term $\mathcal{T}_3$}
	We write this term as
	\begin{flalign*}
		\mathcal{T}_3=\Lambda^s_{x,v}\sum_{j\geq 4}\psi_{j-3}(D_{x,v})\nabla_vf\cdot \phi_j(D_x) A. 
	\end{flalign*}
	We evaluate the $L^p$ norm of the Littlewood-Paley block:
	\begin{flalign*}
		&\left\Vert \phi_k(D_{x,y})\Lambda^s_{x,v}\left(\sum_{j\geq 4}\psi_{j-3}(D_{x,v})\nabla_vf\cdot \phi_j(D_x) A \right)\right\Vert_{L^p_{x,v}}\\&\leq \sum_{j\geq 4} \left\Vert \phi_k(D_{x,y})\Lambda^s_{x,v}\left(\psi_{j-3}(D_{x,v})\nabla_vf\cdot \phi_j(D_x) A \right) \right\Vert_{L^p_{x,v}}
	\end{flalign*}
	Note that the term 
	$$\psi_{j-3}(D_{x,v})\nabla_vf\cdot \phi_j(D_x) A$$
	has its Fourier support inside the annulus $\left\{ 2^{j-2}<\vert (\xi, \mu)\vert < 2^{j+2} \right\}$. Therefore, only $k=j-N, j-M+1, \cdots, j+N$ terms survive with $N$ an independent constant, so we only need to evaluate only for such $(j,k)$ pairs. We have
	\begin{flalign*}
		&\left\Vert \phi_k(D_{x,y})\Lambda^s_{x,v}\left(\psi_{j-3}(D_{x,v})\nabla_vf\cdot \phi_j(D_x) A \right) \right\Vert_{L^p_{x,v}}\leq C\cdot 2^{ks} \left\Vert \psi_{j-3}(D_{x,v})\nabla_vf\cdot \phi_j(D_x) A \right\Vert_{L^p_{x,v}}\\&\leq C\cdot 2^{ks} \Vert \psi_{j-3}(D_{x,v})\nabla_vf \Vert_{L^\frac{p(p+\epsilon')}{\epsilon'}_xL^p_v}  \Vert \phi_j(D_x) A\Vert_{L^{p+\epsilon'}_x}\leq  C\cdot 2^{ks} \cdot2^j\Vert \psi_{j-3}(D_{x,v})f \Vert_{L^\frac{p(p+\epsilon')}{\epsilon'}_xL^p_v}  \Vert \phi_j(D_x) A\Vert_{L^{p+\epsilon'}_x}\\
		&\leq C\cdot 2^{(k+1)s} \Vert \psi_{j-3}(D_{x,v})f \Vert_{L^\frac{p(p+\epsilon')}{\epsilon'}_xL^p_v}  \Vert \phi_j(D_x)\widetilde{\phi}_j(D_x) A\Vert_{L^{p+\epsilon'}_x}\\
		&\leq  C\cdot 2^{(k+1)s} \Vert \psi_{j-3}(D_{x,v})f \Vert_{L^\frac{p(p+\epsilon')}{\epsilon'}_xL^p_v}  \Vert \phi_j(D_x) \nabla_x A\Vert_{L^{p+\epsilon'}_x} \cdot 2^{-j}\\&= C\cdot 2^{ks} \Vert \psi_{j-3}(D_{x,v})f \Vert_{L^\frac{p(p+\epsilon')}{\epsilon'}_xL^p_v}  \Vert \phi_j(D_x) \nabla_x A\Vert_{L^{p+\epsilon'}_x} 
	\end{flalign*}
	where in the last inequality, we have used the Young's convolution inequality to the function $\phi_j(D_x)\widetilde{\phi}_j(D_x) A=  D_x^{-2}\nabla_x\widetilde{\phi}_j\cdot\phi_j(D_x)(D_x) \nabla_xA$ and the fact that the inverse Fourier transform of $\vert\xi\vert^{-2}\xi\widetilde{\phi}_j(\vert\xi\vert)$ is has $L^1$ norm equal to $C\cdot 2^{-j}$. Now, we perform the summation in $(j,k)$ to obtain:
	\begin{flalign*}
			&\sum_{k\in\mathbb{Z}}\left\Vert \phi_k(D_{x,y})\Lambda^s_{x,v}\left(\sum_{j\geq 4}\psi_{j-3}(D_{x,v})\nabla_vf\cdot \phi_j(D_x) A \right)\right\Vert_{L^p_{x,v}}\\&\leq\sum_{j\geq 4} \sum_{k=j-N}^{j+N}C\cdot 2^{ks} \Vert \psi_{j-3}(D_{x,v})f \Vert_{L^\frac{p(p+\epsilon')}{\epsilon'}_xL^p_v}  \Vert \phi_j(D_x) \nabla_x A\Vert_{L^{p+\epsilon'}_x}\\&
			\leq \sum_{j\geq 4}C\cdot 2^{js} \Vert \psi_{j-3}(D_{x,v})f \Vert_{L^\frac{p(p+\epsilon')}{\epsilon'}_xL^p_v}  \Vert \phi_j(D_x) \nabla_x A\Vert_{L^{p+\epsilon'}_x}\\
			&\leq \sum_{j\geq 4} C\cdot 2^{j(\frac{1}{2p}-\epsilon-\frac{2\epsilon'd}{p(p+\epsilon')})} \Vert \Lambda^{-\frac{1}{2p}+\epsilon+\frac{2\epsilon'd}{p(p+\epsilon')}}_{x}f \Vert_{L^\frac{p(p+\epsilon')}{\epsilon'}_xL^p_v}\cdot  2^{js} \Vert \phi_j(D_x) \nabla_x A\Vert_{L^{p+\epsilon'}_x}\\
			&\leq \sum_{j\geq 4} C\cdot 2^{j(\frac{1}{2p}-\epsilon-\frac{2\epsilon'd}{p(p+\epsilon')})} \Vert \Lambda^{-\frac{1}{2p}+\epsilon+\frac{2\epsilon'd}{p(p+\epsilon')}}_{x}f \Vert_{H^{\frac{d}{p+\epsilon'},p}_xL^p_v}\cdot  2^{js} \Vert \phi_j(D_x) \nabla_x A\Vert_{L^{p+\epsilon'}_x}\\
			&\leq \sum_{j\geq 4} C\cdot 2^{-j(\frac{\epsilon'd}{p(p+\epsilon')})} \Vert f \Vert_{H^{s,p}_xL^p_v}\cdot  2^{j(s+\frac{1}{2p}-\epsilon-\frac{\epsilon'd}{p(p+\epsilon')})} \Vert \phi_j(D_x) \nabla_x A\Vert_{L^{p+\epsilon'}_x}\\
			&\leq \sum_{j\geq 4} C\cdot 2^{-j(\frac{\epsilon'd}{p(p+\epsilon')})} \Vert f \Vert_{H^{s,p}_{x,v}}\cdot  2^{j(s+\frac{1}{2p}-\epsilon-\frac{\epsilon'd}{p(p+\epsilon')})} \Vert \phi_j(D_x) \nabla_x A\Vert_{L^{p+\epsilon'}_x},
	\end{flalign*}
    We have used the Young's convolution inequality to the function $$\psi_{j-3}(D_{x,v})f=\psi_{j-3}(D_{x,v})\Lambda_{x}^{\frac{1}{2p}-\epsilon-\frac{2\epsilon'd}{p(p+\epsilon')}}\Lambda_{x}^{-\frac{1}{2p}+\epsilon+\frac{2\epsilon'd}{p(p+\epsilon')}}f$$
	and that the inverse Fourier transform of $\langle \xi \rangle^{\frac{1}{2p}-\epsilon-\frac{2\epsilon'd}{p(p+\epsilon')}}\psi_{j-3}(\xi, \mu)$ has $L^1_{x,v}$ norm bounded by $C2^{j(\frac{1}{2p}-\epsilon-\frac{2\epsilon'd}{p(p+\epsilon')})}$ in the third inequality. This is not so obvious, and we state it as a lemma, whose proof is given in the next section.
	\begin{lemma}
		\label{Aux_Lemma}
		The following inequality
		\begin{flalign*}
			\left\Vert \iint  e^{i(\xi\cdot x+\mu\cdot v)} \langle \xi \rangle^{\frac{1}{2p}-\epsilon-\frac{2\epsilon'd}{p(p+\epsilon')}}\psi_{j-3}(\xi, \mu) d\xi d\mu \right\Vert_{L^1_{x,v}}\leq C2^{(\frac{1}{2p}-\epsilon-\frac{2\epsilon'd}{p(p+\epsilon')})j}
		\end{flalign*} 
		holds where $C$ is independent of $j$.
	\end{lemma}
	We proceed as
	\begin{flalign*}
		\sum_{k\in\mathbb{Z}}&\left\Vert \phi_k(D_{x,y})\Lambda^s_{x,v}\left(\sum_{j\geq 4}\psi_{j-3}(D_{x,v})\nabla_vf\cdot \phi_j(D_x) A \right)\right\Vert_{L^p_{x,v}}\\
		&\leq C\Vert f \Vert_{H^{s,p}_{x,v}}\sum_{j\geq 4}2^{-j(\frac{\epsilon'd}{p(p+\epsilon')})} 2^{j(s+\frac{1}{2p}-\epsilon-\frac{\epsilon'd}{p(p+\epsilon')})} \Vert \phi_j(D_x) \nabla_x A\Vert_{L^{p+\epsilon'}_x}\\
		&\leq  C\Vert f \Vert_{H^{s,p}_{x,v}}\left( \sum_{j\geq4} 2^{-j(\frac{\epsilon'd}{p(p+\epsilon')})\frac{p}{p-1}} \right)^{1-\frac{1}{p+\epsilon'}} \left( \sum_{j\geq 4}2^{jp(s+\frac{1}{2p}-\epsilon-\frac{\epsilon'd}{p(p+\epsilon')})} \Vert \phi_j(D_x) \nabla_x A\Vert^p_{L^p_x} \right)^\frac{1}{p+\epsilon'}\\&\leq C\Vert f \Vert_{H^{s,p}_{x,v}} \Vert \nabla_x A \Vert_{(F^{s+\frac{1}{2p}-\epsilon-\frac{\epsilon'd}{p(p+\epsilon')}}_{p+\epsilon',p+\epsilon'})_x}\\
		&\leq C\Vert f \Vert_{H^{s,p}_{x,v}} \Vert \nabla_x A \Vert_{H^{s+\frac{1}{2p}-\epsilon-\frac{\epsilon'd}{p(p+\epsilon')}, p+\epsilon'}_x}\leq C\Vert f \Vert_{H^{s,p}_{x,v}} \Vert A \Vert_{H^{s+1+\frac{1}{2p}-\epsilon-\frac{\epsilon'd}{p(p+\epsilon')}, p+\epsilon'}_x}
	\end{flalign*}
	
	\subsection{Term $\mathcal{T}_4$} 
	We write it as
	\begin{flalign*}
		\mathcal{T}_4= \Lambda^s_{x,v}\left( \sum_{\substack{\max(j, k)\geq 4 \\ \vert j-k\vert \leq 2}} \phi_j(D_{x, v})\nabla_v f\cdot \phi_k(D_x)A + \psi_3(D_{x,v})\nabla_v f\cdot \psi_3(D_x)A\right).
	\end{flalign*}
	We estimate each term. First, since the Fourier support of the function $\phi_j(D_{x, v})\nabla_v f\cdot \phi_k(D_x)A$ lies inside the ball $\vert(\xi, \mu)\vert\leq 2^{2+\max(j,k)}$ ,we have
	\begin{flalign*}
		&\Vert \Lambda^s_{x,v}\left( \phi_j(D_{x, v})\nabla_v f\cdot \phi_k(D_x)A \right) \Vert_{L^p_{x,v}} \leq C\cdot 2^{\max(j,k)s} \Vert  \phi_j(D_{x, v})\nabla_v f\cdot \phi_k(D_x)A \Vert_{L^p_{x,v}}\\& \leq C\cdot 2^{\max(j,k)s} \Vert \phi_j(D_{x, v})\nabla_v f \Vert_{L^p_{x,v}}\Vert \phi_k(D_x)A \Vert_{L^\infty_{x}} \\&\leq C\cdot 2^{\max(j,k)s} \cdot 2^j \Vert \phi_j(D_{x, v}) f \Vert_{L^p_{x,v}}  2^{\frac{k}{p+\epsilon'}} \cdot\Vert \phi_k(D_x)A \Vert_{L^{p+\epsilon'}_{x}}\\
		&\leq C\cdot 2^{\max(j,k)s} \cdot 2^{j-k} \Vert \phi_j(D_{x, v}) f \Vert_{L^p_{x,v}}  2^{\frac{k}{p+\epsilon'}} \cdot\Vert \phi_k(D_x)\nabla_xA \Vert_{L^{p+\epsilon'}_{x}}
	\end{flalign*}
	and therefore, we have
	\begin{flalign*}
		&\sum_{\substack{\max(j, k)\geq 4 \\ \vert j-k\vert \leq 2}}\Vert \Lambda^s_{x,v}\left( \phi_j(D_{x, v})\nabla_v f\cdot \phi_k(D_x)A \right) \Vert_{L^p_{x,v}} \\&\leq C \sum_{l=-2}^2\sum_{j=2}^\infty 2^{-j(s-\frac{1}{2p}-\epsilon-\frac{\epsilon'd}{p(p+\epsilon')})}  \left(2^{js}\Vert \phi_j(D_{x, v}) f \Vert_{L^p_{x,v}}\right)\cdot\left(2^{(j+l)(s+1+\frac{1}{2p}-\epsilon-\frac{\epsilon'd}{p(p+\epsilon')})} \Vert \phi_{j+l}(D_x) A \Vert_{L^{p+\epsilon'}_{x}} \right)\\&\leq  C \sum_{l=-2}^2 \left( \sum_{j=2}^\infty   2^{jps}\Vert \phi_j(D_{x, v}) f \Vert^p_{L^p_{x,v}} \right)^\frac{1}{p} \cdot \left(\sum_{j=2}^\infty 2^{(j+l)p(s+1+\frac{1}{2p}-\epsilon-\frac{\epsilon'd}{p(p+\epsilon')})} \Vert \phi_{j+l}(D_x) A \Vert^{p+\epsilon'}_{L^{p+\epsilon'}_{x}} \right)^\frac{1}{p+\epsilon'} \\&\times\Vert 2^{-j(s+\frac{1}{2p}-\epsilon-\frac{\epsilon'd}{p(p+\epsilon')})} \Vert_{l^\frac{1}{1-\frac{1}{p}-\frac{1}{p+\epsilon'}}(\mathbb{N})} \\&\leq C\Vert f \Vert_{H^{s,p}_{x,v}}\Vert A \Vert_{H^{s+1+\frac{1}{2p}-\epsilon-\frac{\epsilon'd}{p(p+\epsilon')},p+\epsilon'}_{x}}.
	\end{flalign*}
	The only term left to estimate is
	\begin{flalign*}
		\Lambda^s_{x,v}\psi_3(D_{x,v})\nabla_v f\cdot \psi_3(D_x)A.
	\end{flalign*}
	This can be estimated as follows:
	\begin{flalign*}
		\Vert \Lambda^s_{x,v}\psi_3(D_{x,v})\nabla_v f\cdot \psi_3(D_x)A \Vert_{L^p_{x,v}} &\leq C\Vert \psi_3(D_{x,v})\nabla_v f\cdot \psi_3(D_x)A \Vert_{L^p_{x,v}} \leq C \Vert \psi_3(D_{x,v})\nabla_v f  \Vert_{L^p_{x,v}}\Vert \psi_3(D_x)A \Vert_{L^\infty_x}\\
		&\leq C \Vert \psi_3(D_{x,v}) f  \Vert_{H^{s,p}_{x,v}}\Vert \psi_3(D_x)A \Vert_{H^{s+\frac{1}{2p}-\epsilon-\frac{\epsilon'd}{p(p+\epsilon')},p+\epsilon'}_x}  \\&\leq C\Vert  f  \Vert_{H^{s, p}_{x,v}}\Vert A \Vert_{H^{s+\frac{1}{2p}-\epsilon-\frac{\epsilon'd}{p(p+\epsilon')},p+\epsilon'}_x}
	\end{flalign*}
	and we are done.
	\section{Proof of Lemma \ref{Aux_Lemma}}
	\label{Aux_Lemma_Proof}
	In this section, we prove Lemma \ref{Aux_Lemma}.
	\begin{proof}[proof of Lemma \ref{Aux_Lemma}]
		Let $\zeta_j(\cdot)=\zeta(2^{j}\cdot)$, where $\zeta\in C^\infty_c(\mathbb{R}^d)$ is a radial function such that $$\zeta(\xi)\zeta(\mu)\psi(\xi, \mu)=\psi(\xi, \mu)$$ for all $(\xi, \mu)$. We write
		\begin{flalign*}
			\iint  e^{i(\xi\cdot x+\mu\cdot v)} \langle \xi \rangle^{\frac{1}{2p}-\epsilon-\frac{2\epsilon'd}{p(p+\epsilon')}}\psi_{j}(\xi, \mu) d\xi d\mu = \iint  e^{i(\xi\cdot x+\mu\cdot v)} \langle \xi \rangle^{\frac{1}{2p}-\epsilon-\frac{2\epsilon'd}{p(p+\epsilon')}}\zeta_{j}(\xi)\zeta_{j}(\mu)\psi_{j}(\xi, \mu) d\xi d\mu,
		\end{flalign*}
		which a (up to multiplication by a constant) convolution of two functions $\int e^{i(\xi\cdot x+\mu\cdot v)} \psi_{j}(\xi, \mu) d\xi$ and $\int e^{i(\xi\cdot x+\mu\cdot v)} \langle \xi \rangle^{\frac{1}{2p}-\epsilon-\frac{2\epsilon'd}{p(p+\epsilon')}}\zeta_{j}(\xi)\zeta_{j}(\mu) d\xi$. Therefore, it suffices to bound $L^1$ norms of these terms. 
		
		The first one is just $2^{2jd}\check{\psi}(2^jd x, 2^jd v)$, whose $L^1$ norm is finite and independent of $j$.
		
		For the second one, we have
		\begin{flalign*}
			&\iint e^{i(\xi\cdot x+\mu\cdot v)} (1+\vert\xi\vert^2)^{\frac{\frac{1}{2p}-\epsilon-\frac{2\epsilon'd}{p(p+\epsilon')}}{2}}\zeta_{j}(\xi)\zeta_j(\mu) d\xi d\mu = \iint e^{i(\xi\cdot x+\mu\cdot v)} (1+\vert\xi\vert^2)^\frac{\frac{1}{2p}-\epsilon-\frac{2\epsilon'd}{p(p+\epsilon')}}{2}\zeta(2^{-j}\xi)\zeta(2^{-j}\mu) d\xi \\&= 2^{2jd}\iint e^{2^j i(\eta\cdot x+\nu\cdot v)} (1+\vert2^j\eta\vert^2)^{\frac{\frac{1}{2p}-\epsilon-\frac{2\epsilon'd}{p(p+\epsilon')}}{2}}\zeta(\eta)\zeta(\nu) d\eta d\nu \\&= 2^{(\frac{1}{2p}-\epsilon-\frac{2\epsilon'd}{p(p+\epsilon')})j}\cdot 2^{2jd}\iint e^{2^ji(\eta\cdot x_\nu\cdot v)} (2^{-2j}+\vert\eta\vert^2)^{\frac{\frac{1}{2p}-\epsilon-\frac{2\epsilon'd}{p(p+\epsilon')}}{2}}\zeta(\eta)\zeta(\nu) d\eta d\nu.
		\end{flalign*}
		It remains to obtain a uniform bound on $$\left\Vert 2^{2jd}\iint e^{2^j i(\eta\cdot x+\nu\cdot v)} (2^{-2j}+\vert\eta\vert^2)^{\frac{\frac{1}{2p}-\epsilon-\frac{2\epsilon'd}{p(p+\epsilon')}}{2}}\zeta(\eta)\zeta(\nu) d\eta d\nu \right\Vert_{L^1_{x,v}}.$$ To do this, we manipulate it as follows:
		\begin{flalign*}
			&\left\Vert 2^{2jd}\iint e^{2^j i(\eta\cdot x+\nu\cdot v)} (2^{-2j}+\vert\eta\vert^2)^{\frac{\frac{1}{2p}-\epsilon-\frac{2\epsilon'd}{p(p+\epsilon')}}{2}}\zeta(\eta)\zeta(\nu) d\eta d\nu \right\Vert_{L^1_{x,v}}\\&=\left\Vert \iint e^{i(\eta\cdot x+\nu\cdot v)} (2^{-2j}+\vert\eta\vert^2)^{\frac{\frac{1}{2p}-\epsilon-\frac{2\epsilon'd}{p(p+\epsilon')}}{2}}\zeta(\eta)\zeta(\nu) d\eta d\nu \right\Vert_{L^1_x}\\&=\left\Vert \iint e^{i(\eta\cdot x+\nu\cdot v)} \frac{(2^{-2j}+\vert\eta\vert^2)^{\frac{\frac{1}{2p}-\epsilon-\frac{2\epsilon'd}{p(p+\epsilon')}}{2}}}{(1+\vert \eta\vert^2)^{d}}\cdot(1+\vert \eta\vert^2)^{d}\zeta(\eta)\zeta(\nu) d\eta d\nu \right\Vert_{L^1_x}\\&\leq C\left\Vert \int e^{i\eta\cdot x} \frac{(2^{-2j}+\vert\eta\vert^2)^{\frac{\frac{1}{2p}-\epsilon-\frac{2\epsilon'd}{p(p+\epsilon')}}{2}}}{(1+\vert \eta\vert^2)^{d}} d\eta \right\Vert_{L^1_x}\left\Vert \int e^{i\eta\cdot x} (1+\vert \eta\vert^2)^{d}\zeta(\eta) d\eta \right\Vert_{L^1_x}\\&\leq C\left\Vert \int e^{i\eta\cdot x} \frac{(2^{-2j}+\vert\eta\vert^2)^{\frac{\frac{1}{2p}-\epsilon-\frac{2\epsilon'd}{p(p+\epsilon')}}{2}}}{(1+\vert \eta\vert^2)^{d}} d\eta \right\Vert_{L^1_x}\left\Vert (1-\Delta_x)^{d} \check{\zeta} \right\Vert_{L^1_x}.
		\end{flalign*}
		All we have to do now is to bound the $L^1$ norm of the following Kernel, uniformly in $j$:
		\begin{flalign*}
			K_{j}(x)=\int_{\mathbb{R}^d} e^{i\eta\cdot x} \frac{(2^{-2j}+\vert\eta\vert^2)^{\frac{\frac{1}{2p}-\epsilon-\frac{2\epsilon'd}{p(p+\epsilon')}}{2}}}{(1+\vert \eta\vert^2)^{d}} d\eta=\int_{\mathbb{R}^d} e^{i\eta\cdot x} \frac{(2^{-2j}+\vert\eta\vert^2)^{\frac{\alpha}{2}}}{(1+\vert \eta\vert^2)^{d}} d\eta
		\end{flalign*}
		where $\alpha=\frac{1}{2p}-\epsilon-\frac{2\epsilon'd}{p(p+\epsilon')}>0$.
		To prove this, we use another Littlewood-Paley decomposition. Let $\chi$ be a smooth function supported on $[0, \frac{8}{3}]$, taking values between $0$ and $1$, and $\chi(\eta)=1$ for $\vert \eta\vert\leq  \frac{3}{2}$. Then, we decompose the kernel as follows:
		\begin{flalign*}
			K_{j}(x)&=\int_{\mathbb{R}^d} e^{i\eta\cdot x} \frac{(2^{-2j}+\vert\eta\vert^2)^{\frac{\alpha}{2}}}{(1+\vert \eta\vert^2)^{d}}\chi(2^{j-2}\eta) d\eta+\sum_{k=0}^{j-3}\int_{\mathbb{R}^d} e^{i\eta\cdot x} \frac{(2^{-2j}+\vert\eta\vert^2)^{\frac{\alpha}{2}}}{(1+\vert \eta\vert^2)^{d}} \left( \chi(2^k\eta)-\chi(2^{k+1}\eta) \right)d\eta \\&+\int_{\mathbb{R}^d} e^{i\eta\cdot x} \frac{(2^{-2j}+\vert\eta\vert^2)^{\frac{\alpha}{2}}}{(1+\vert \eta\vert^2)^{d}} \left( 1-\chi(\eta) \right) d\eta.
		\end{flalign*}
		First, let us consider the first kernel:
		\begin{flalign*}
			\int_{\mathbb{R}^d} e^{i\eta\cdot x} \frac{(2^{-2j}+\vert\eta\vert^2)^{\frac{\alpha}{2}}}{(1+\vert \eta\vert^2)^{d}}\chi(2^{j-2}\eta) d\eta.
		\end{flalign*}
		Its $L^1$ norm is bounded as follows:
		\begin{flalign*}
			&\left\Vert\int_{\mathbb{R}^d} e^{i\eta\cdot x} \frac{(2^{-2j}+\vert\eta\vert^2)^{\frac{\alpha}{2}}}{(1+\vert \eta\vert^2)^{d}}\chi(2^{j-2}\eta) d\eta\right\Vert_{L^1(\mathbb{R}^d)}	\\&\leq C \left\Vert\int_{\mathbb{R}^d} e^{i\eta\cdot x} \frac{1}{(1+\vert \eta\vert^2)^{d}} d\eta\right\Vert_{L^1(\mathbb{R}^d)}\left\Vert\int_{\mathbb{R}^d} e^{i\eta\cdot x} (2^{-2j}+\vert\eta\vert^2)^{\frac{\alpha}{2}}\chi(2^{j-2}\eta) d\eta\right\Vert_{L^1(\mathbb{R}^d)}\\&\leq C 2^{-j\alpha} \left\Vert\int_{\mathbb{R}^d} e^{i\eta\cdot x} \frac{1}{(1+\vert \eta\vert^2)^{d}} d\eta\right\Vert_{L^1(\mathbb{R}^d)}\left\Vert\int_{\mathbb{R}^d} e^{i\eta\cdot x} (1+\vert\eta\vert^2)^{\frac{\alpha}{2}}\chi(4\eta) d\eta\right\Vert_{L^1(\mathbb{R}^d)}
		\end{flalign*}
		The term $$ \left\Vert\int_{\mathbb{R}^d} e^{i\eta\cdot x} \frac{1}{(1+\vert \eta\vert^2)^{d}} d\eta\right\Vert_{L^1(\mathbb{R}^d)}$$ is finite by the following analysis of the Kernel:
		\begin{flalign*}
			\left\vert \int_{\mathbb{R}^d} e^{i\eta\cdot x} \frac{1}{(1+\vert \eta\vert^2)^{d}} d\eta \right\vert \leq \int_{\mathbb{R}^d}  \frac{1}{(1+\vert \eta\vert^2)^{d}} d\eta =C<\infty,
		\end{flalign*}
		and
		\begin{flalign*}
			&\left\vert \int_{\mathbb{R}^d} e^{i\eta\cdot x} \frac{1}{(1+\vert \eta\vert^2)^{d}} d\eta \right\vert = \left\vert \int_{\mathbb{R}^d} \left( \frac{x\cdot\nabla_\eta}{i\vert x\vert^2} \right)^{d+1}(e^{i\eta\cdot x}) \frac{1}{(1+\vert \eta\vert^2)^{d}} d\eta \right\vert = \left\vert \int_{\mathbb{R}^d} e^{i\eta\cdot x} \left( \frac{x\cdot\nabla_\eta}{i\vert x\vert^2} \right)^{d+1} \left(\frac{1}{(1+\vert \eta\vert^2)^{d}}\right) d\eta \right\vert \\& \leq C\frac{1}{\vert x\vert^{d+1}} \int_{\mathbb{R}^d} \left\vert \nabla_{\eta}^{d+1}\left(\frac{1}{(1+\vert \eta\vert^2)^{d}}\right) \right\vert d\eta = C<\infty,
		\end{flalign*}
		and therefore, satisfies 
		\begin{flalign*}
			\left\vert \int_{\mathbb{R}^d} e^{i\eta\cdot x} \frac{1}{(1+\vert \eta\vert^2)^{d}} d\eta \right\vert \leq \frac{C}{1+\vert x\vert^{d+1}},
		\end{flalign*}
		which is integrable.
		
		The other term
		\begin{flalign*}
			\left\Vert\int_{\mathbb{R}^d} e^{i\eta\cdot x} (1+\vert\eta\vert^2)^{\frac{\alpha}{2}}\chi(4\eta) d\eta\right\Vert_{L^1(\mathbb{R}^d)}
		\end{flalign*}
		is also finite, since $$\int_{\mathbb{R}^d} e^{i\eta\cdot x} (1+\vert\eta\vert^2)^{\frac{\alpha}{2}}\chi(4\eta) d\eta$$ is a Fourier transform of a Schwartz function and thus in $L^1(\mathbb{R}^d)$.
		
		The next step is to analyze the final term:
		\begin{flalign*}
			\int_{\mathbb{R}^d} e^{i\eta\cdot x} \frac{(2^{-2j}+\vert\eta\vert^2)^{\frac{\alpha}{2}}}{(1+\vert \eta\vert^2)^{d}} \left( 1-\chi(\eta) \right) d\eta
		\end{flalign*}
		Again, we have
		\begin{flalign*}
			&\left\vert \int_{\mathbb{R}^d} e^{i\eta\cdot x} \frac{(2^{-2j}+\vert\eta\vert^2)^{\frac{\alpha}{2}}}{(1+\vert \eta\vert^2)^{d}} \left( 1-\chi(\eta) \right) d\eta \right\vert \leq \int_{\mathbb{R}^d}  \frac{(2^{-2j}+\vert\eta\vert^2)^{\frac{\alpha}{2}}}{(1+\vert \eta\vert^2)^{d}} \left( 1-\chi(\eta) \right) d\eta  \leq \int_{\mathbb{R}^d}  (1+\vert \eta\vert^2)^{\frac{\alpha}{2}-d} d\eta \\&=C<\infty,
		\end{flalign*}
		and
		\begin{flalign*}
			&\left\vert \int_{\mathbb{R}^d} e^{i\eta\cdot x} \frac{(2^{-2j}+\vert\eta\vert^2)^{\frac{\alpha}{2}}}{(1+\vert \eta\vert^2)^{d}} \left( 1-\chi(\eta) \right) d\eta \right\vert = \left\vert \int_{\mathbb{R}^d} e^{i\eta\cdot x}\left( \frac{x\cdot\nabla_\eta}{i\vert x\vert^2} \right)^{d+1}\left( \frac{(2^{-2j}+\vert\eta\vert^2)^{\frac{\alpha}{2}}}{(1+\vert \eta\vert^2)^{d}} \left( 1-\chi(\eta) \right) \right) d\eta\right\vert\\&\leq \int_{\mathbb{R}^d} \left\vert \left( \frac{x\cdot\nabla_\eta}{i\vert x\vert^2} \right)^{d+1}\left( \frac{(2^{-2j}+\vert\eta\vert^2)^{\frac{\alpha}{2}}}{(1+\vert \eta\vert^2)^{d}} \left( 1-\chi(\eta) \right) \right) \right\vert d\eta \leq \frac{C}{\vert x\vert^{d+1}} \int_{\mathbb{R}^d} \left\vert \nabla^{d+1}_\eta \left( \frac{(2^{-2j}+\vert\eta\vert^2)^{\frac{\alpha}{2}}}{(1+\vert \eta\vert^2)^{d}} \left( 1-\chi(\eta) \right) \right)  \right\vert d\eta \\&\leq \frac{C}{\vert x\vert^{d+1}} \int_{\vert\eta\vert\geq 1} \sum_{m=0}^{d+1} \left\vert \nabla^m_\eta \left( (2^{-2j}+\vert\eta\vert^2)^{\frac{\alpha}{2}} \right) \right\vert \left\vert \nabla^{d+1-m}_\eta \left( \frac{1-\chi(\eta)}{(1+\vert\eta\vert^2)^d} \right) \right\vert d\eta
		\end{flalign*}
		One observes that the components of 
		\begin{flalign*}
			\nabla^m_\eta \left( (2^{-2j}+\vert\eta\vert^2)^{\frac{\alpha}{2}} \right)
		\end{flalign*}
		are of the form
		\begin{flalign*}
			\sum_{i=\lceil \frac{m}{2} \rceil}^m p_{m,i}(\eta) (2^{-2j}+\vert \eta\vert^2)^{\frac{\alpha}{2}-i}
		\end{flalign*}
		where $p_{m, i}$ are homogeneous polynomial in $\eta$ of order $2i-m$, which is independent of $j$. Since $(2^{-2j}+\vert \eta\vert^2)^{\frac{\alpha}{2}-i}\leq \vert \eta\vert^{\alpha-2i} $ for $i\geq 1$ and $(2^{-2j}+\vert \eta\vert^2)^{\frac{\alpha}{2}}\leq \vert (1+\vert \eta\vert^2)\vert^{\frac{\alpha}{2}}$, we see that
		\begin{flalign*}
			\left\vert \nabla^m_\eta \left( (2^{-2j}+\vert\eta\vert^2)^{\frac{\alpha}{2}} \right) \right\vert \leq C \vert \eta\vert^{\alpha-m}
		\end{flalign*}
		for $m\geq 1$. Thus, we have
		\begin{flalign*}
			\left\vert \int_{\mathbb{R}^d} e^{i\eta\cdot x} \frac{(2^{-2j}+\vert\eta\vert^2)^{\frac{\alpha}{2}}}{(1+\vert \eta\vert^2)^{d}} \left( 1-\chi(\eta) \right) d\eta \right\vert &\leq \frac{C}{\vert x\vert^{m+1}} \int_{\vert\eta\vert\geq 1} \sum_{m=0}^{d+1} \vert\eta\vert^{\alpha-m} \left\vert \nabla^{d+1-m}_\eta \left( \frac{1-\chi(\eta)}{(1+\vert\eta\vert^2)^d} \right) \right\vert d\eta
		\end{flalign*}
		We see that the integral on the right hand side is finite. Thus, we see that
		\begin{flalign*}
			\left\vert \int_{\mathbb{R}^d} e^{i\eta\cdot x} \frac{(2^{-2j}+\vert\eta\vert^2)^{\frac{\alpha}{2}}}{(1+\vert \eta\vert^2)^{d}} \left( 1-\chi(\eta) \right) d\eta \right\vert \leq \frac{C}{1+\vert x\vert^{m+1}}
		\end{flalign*}
		whose right hand side is integrable. 
		
		Finally, we analyze terms in the summand:
		\begin{flalign*}
			\int_{\mathbb{R}^d} e^{i\eta\cdot x} \frac{(2^{-2j}+\vert\eta\vert^2)^{\frac{\alpha}{2}}}{(1+\vert \eta\vert^2)^{d}} \left( \chi(2^k\eta)-\chi(2^{k+1}\eta) \right)d\eta
		\end{flalign*}
		for $0\leq k\leq j-3$.
		
		First, the $L^\infty$-bound can be obtained\footnote{Note that the estimate is very crude, but the reason for this setting appears in the proof of decay.}:
		\begin{flalign*}
			\left\vert \int_{\mathbb{R}^d} e^{i\eta\cdot x} \frac{(2^{-2j}+\vert\eta\vert^2)^{\frac{\alpha}{2}}}{(1+\vert \eta\vert^2)^{d}} \left( \chi(2^k\eta)-\chi(2^{k+1}\eta) \right)d\eta \right\vert \leq \int_{\frac{3}{4}\cdot2^{-k}\leq \vert\eta\vert\leq \frac{8}{3}\cdot2^{-k}}  (2^{-2j}+\vert\eta\vert^2)^{\frac{\alpha}{2}} d\eta=C 2^{-\frac{k\alpha}{2}}
		\end{flalign*}
		Second, we obtain:
		\begin{flalign*}
			&\left\vert \int_{\mathbb{R}^d} e^{i\eta\cdot x} \frac{(2^{-2j}+\vert\eta\vert^2)^{\frac{\alpha}{2}}}{(1+\vert \eta\vert^2)^{d}} \left( \chi(2^k\eta)-\chi(2^{k+1}\eta) \right)d\eta \right\vert \\&= \left\vert \int_{\mathbb{R}^d} e^{i\eta\cdot x} \left( \frac{x\cdot\nabla_\eta}{i\vert x\vert^2} \right)^{d+1}\left(\frac{(2^{-2j}+\vert\eta\vert^2)^{\frac{\alpha}{2}}}{(1+\vert \eta\vert^2)^{d}} \left( \chi(2^k\eta)-\chi(2^{k+1}\eta) \right)\right)d\eta \right\vert\\
			&\leq \frac{C}{\vert x\vert^{d+1}}\int_{\mathbb{R}^d} \left\vert \nabla^{d+1}_{\eta} \left( \frac{(2^{-2j}+\vert\eta\vert^2)^{\frac{\alpha}{2}}}{(1+\vert \eta\vert^2)^{d}} \left( \chi(2^k\eta)-\chi(2^{k+1}\eta) \right) \right)  \right\vert d\eta \\
			&\leq \frac{C}{\vert x\vert^{d+1}}\int_{\mathbb{R}^d}  \sum_{l+m+n=d+1} \left\vert \nabla^l_\eta\left( 2^{-2j}+\vert\eta\vert^2 \right)^\frac{\alpha}{2} \right\vert\left\vert \nabla^{m}_\eta \left(\frac{1}{(1+\vert \eta\vert^2)^d}\right) \right\vert \left\vert \nabla_\eta^n\left( \chi(2^k\eta)-\chi(2^{k+1}\eta) \right)  \right\vert    d\eta
		\end{flalign*}
		Now, we observe the following facts:
		\begin{enumerate}
			\item The integrand is supported in $\frac{3}{4}\cdot2^{-k}\leq \vert\eta\vert\leq \frac{8}{3}\cdot2^{-k}$.
			\item In the support,
			\begin{flalign*}
				\left\vert \nabla^{m}_\eta \left(\frac{1}{(1+\vert \eta\vert^2)^d}\right) \right\vert
			\end{flalign*}
			is just bounded, uniformly in $0\leq k\leq j-3$.
			\item In the support,
			\begin{flalign*}
				\left\vert\nabla^l_\eta\left( 2^{-2j}+\vert\eta\vert^2 \right)^\frac{\alpha}{2} \right\vert \leq C_l 2^{k(l-\alpha)}.
			\end{flalign*}
			and
			\begin{flalign*}
				\left\vert \nabla_\eta^n\left( \chi(2^k\eta)-\chi(2^{k+1}\eta) \right)  \right\vert \leq C_n 2^{kn}
			\end{flalign*}
		\end{enumerate}
		Therefore, we have that
		\begin{flalign*}
			&\left\vert \int_{\mathbb{R}^d} e^{i\eta\cdot x} \frac{(2^{-2j}+\vert\eta\vert^2)^{\frac{\alpha}{2}}}{(1+\vert \eta\vert^2)^{d}} \left( \chi(2^k\eta)-\chi(2^{k+1}\eta) \right)d\eta \right\vert \\
			&\leq \frac{C}{\vert x\vert^{d+1}}\int_{\mathbb{R}^d}  \sum_{l+m+n=d+1} \left\vert \nabla^l_\eta\left( 2^{-2j}+\vert\eta\vert^2 \right)^\frac{\alpha}{2} \right\vert\left\vert \nabla^{m}_\eta \left(\frac{1}{(1+\vert \eta\vert^2)^d}\right) \right\vert \left\vert \nabla_\eta^n\left( \chi(2^k\eta)-\chi(2^{k+1}\eta) \right)  \right\vert    d\eta\\
			&\leq \frac{C}{\vert x\vert^{d+1}}\int_{\frac{3}{4}\cdot2^{-k}\leq \vert\eta\vert\leq \frac{8}{3}\cdot2^{-k}}  \sum_{l+m+n=d+1} C_lC_n 2^{k(l+n-\alpha)} d\eta \leq \frac{C}{\vert x\vert^{d+1}} \cdot 2^{k(1-\alpha)}
		\end{flalign*}
		where $C$ is independent of $k$.
		
		Third, we obtain:
		\begin{flalign*}
			&\left\vert \int_{\mathbb{R}^d} e^{i\eta\cdot x} \frac{(2^{-2j}+\vert\eta\vert^2)^{\frac{\alpha}{2}}}{(1+\vert \eta\vert^2)^{d}} \left( \chi(2^k\eta)-\chi(2^{k+1}\eta) \right)d\eta \right\vert=\left\vert \int_{\mathbb{R}^d} e^{i\eta\cdot x}\left( \frac{x\cdot\nabla_\eta}{i\vert x\vert^2} \right)^{d}\left( \frac{(2^{-2j}+\vert\eta\vert^2)^{\frac{\alpha}{2}}}{(1+\vert \eta\vert^2)^{d}} \left( 1-\chi(\eta) \right) \right) d\eta\right\vert\\&
			\leq \frac{C}{\vert x\vert^{d}}\int_{\mathbb{R}^d}  \sum_{l+m+n=d} \left\vert \nabla^l_\eta\left( 2^{-2j}+\vert\eta\vert^2 \right)^\frac{\alpha}{2} \right\vert\left\vert \nabla^{m}_\eta \left(\frac{1}{(1+\vert \eta\vert^2)^d}\right) \right\vert \left\vert \nabla_\eta^n\left( \chi(2^k\eta)-\chi(2^{k+1}\eta) \right)  \right\vert    d\eta
		\end{flalign*}
		and by the same analysis we have done in the above, we see that the right hand side is bounded by
		\begin{flalign*}
			\frac{C}{\vert x\vert^{d}} \cdot 2^{-k\alpha}
		\end{flalign*}
		By combining the second and the third argument, we see that the following holds:
		\begin{flalign*}
			\vert x\vert^{d+\frac{\alpha}{2}}\left\vert \int_{\mathbb{R}^d} e^{i\eta\cdot x} \frac{(2^{-2j}+\vert\eta\vert^2)^{\frac{\alpha}{2}}}{(1+\vert \eta\vert^2)^{d}} \left( \chi(2^k\eta)-\chi(2^{k+1}\eta) \right)d\eta \right\vert \leq C 2^{-\frac{k\alpha}{2}}
		\end{flalign*}
		where $C$ is independent of $k$.
		Therefore, we see that
		\begin{flalign*}
			\left\Vert \int_{\mathbb{R}^d} e^{i\eta\cdot x} \frac{(2^{-2j}+\vert\eta\vert^2)^{\frac{\alpha}{2}}}{(1+\vert \eta\vert^2)^{d}} \left( \chi(2^k\eta)-\chi(2^{k+1}\eta) \right)d\eta \right\Vert_{L^1(\mathbb{R}^d)} \leq \int_{\mathbb{R}^d}\frac{C2^{-\frac{k\alpha}{2}}}{1+\vert x\vert^{d+\frac{\alpha}{2}}}d\eta\leq C2^{-\frac{k\alpha}{2}}.
		\end{flalign*}
		Using the above results, finally we can obtain a uniform bound for the $L^1$-norm of $K_j$:
		\begin{flalign*}
			\Vert K_j\Vert_{L^1(\mathbb{R}^d)}&\leq \left\Vert\int_{\mathbb{R}^d} e^{i\eta\cdot x} \frac{(2^{-2j}+\vert\eta\vert^2)^{\frac{\alpha}{2}}}{(1+\vert \eta\vert^2)^{d}}\chi(2^{j-2}\eta) d\eta\right\Vert_{L^1(\mathbb{R}^d)}\\&+\sum_{k=0}^{j-3}\left\Vert\int_{\mathbb{R}^d} e^{i\eta\cdot x} \frac{(2^{-2j}+\vert\eta\vert^2)^{\frac{\alpha}{2}}}{(1+\vert \eta\vert^2)^{d}} \left( \chi(2^k\eta)-\chi(2^{k+1}\eta) \right)d\eta\right\Vert_{L^1(\mathbb{R}^d)} \\&+\left\Vert\int_{\mathbb{R}^d} e^{i\eta\cdot x} \frac{(2^{-2j}+\vert\eta\vert^2)^{\frac{\alpha}{2}}}{(1+\vert \eta\vert^2)^{d}} \left( 1-\chi(\eta) \right) d\eta\right\Vert_{L^1(\mathbb{R}^d)}\\
			&\leq C2^{-j\alpha}+C\sum_{k=0}^{j-3} 2^{-\frac{k\alpha}{2}} +C \leq C\sum_{k=0}^\infty 2^{-\frac{k\alpha}{2}}<\infty.
		\end{flalign*} 
	\end{proof}

	\bibliographystyle{siam}
	\bibliography{VP-low-regularity-2}

\begin{thebibliography}{10}

\bibitem{Ars}
{\sc A.~A. Arsenev}, {\em Existence in the large of a weak solution of
  {V}lasov's system of equations}, \v Z. Vy\v cisl. Mat i Mat. Fiz., 15 (1975),
  pp.~136--147, 276.

\bibitem{BGP}
{\sc F.~c. Bouchut, F.~c. Golse, and M.~Pulvirenti}, {\em Kinetic equations and
  asymptotic theory}, vol.~4 of Series in Applied Mathematics (Paris),
  Gauthier-Villars, \'Editions Scientifiques et M\'edicales Elsevier, Paris,
  2000.
\newblock Edited and with a foreword by Beno\^it Perthame and Laurent
  Desvillettes.

\bibitem{Bounded}
{\sc L.~Cesbron}, {\em Global well posedness of vlasov-poisson-type systems in
  bounded domians}, Analysis and PDE, 16 (2023), pp.~2465--2494.

\bibitem{Yudovich}
{\sc C.~S. Gianluca~Crippa, Marco~Inversi and G.~Stefani}, {\em Existence and
  stability of weak solutions of the vlasov-poisson system in localized
  yudovich spaces}, Nonlinearity, 37 (2024).

\bibitem{Glassey96}
{\sc R.~T. Glassey}, {\em The {C}auchy problem in kinetic theory}, Society for
  Industrial and Applied Mathematics (SIAM), Philadelphia, PA, 1996.

\bibitem{Grafakos}
{\sc L.~Grafakos}, {\em Modern Fourier Analysis}, Graduate Texts in
  Mathematics, Springer, 3rd~ed., 2010.

\bibitem{Horst}
{\sc E.~Horst}, {\em On the classical solutions of the initial value problem
  for the unmodified nonlinear {V}lasov equation. {I}. {G}eneral theory}, Math.
  Methods Appl. Sci., 3 (1981), pp.~229--248.

\bibitem{Horst2}
\leavevmode\vrule height 2pt depth -1.6pt width 23pt, {\em On the classical
  solutions of the initial value problem for the unmodified nonlinear {V}lasov
  equation. {II}. {S}pecial cases}, Math. Methods Appl. Sci., 4 (1982),
  pp.~19--32.

\bibitem{Convex}
{\sc H.~J. Hwang, J.~Jung, and J.~J.~L. Vel\'azquez}, {\em On global existence
  of classical solutions for the {V}lasov-{P}oisson system in convex bounded
  domains}, Discrete Contin. Dyn. Syst., 33 (2013), pp.~723--737.

\bibitem{Prev}
{\sc I.-J. Jeong and S.~Tae}, {\em Low regularity sobolev well-posedness for
  {V}lasov--{P}oisson}.

\bibitem{Besov}
{\sc C.~H. Jingchun~Chen}, {\em Vlasov-poisson equation in besov space},
  Taiwanese Journal of Mathematics, 1 (2022).

\bibitem{W.F.S._2}
\leavevmode\vrule height 2pt depth -1.6pt width 23pt, {\em Vlasov-poisson
  equations in {${H}^{s, p}(w)$} space}, Michigan Mathematical Journal, 73
  (2022).

\bibitem{W.F.S._1}
\leavevmode\vrule height 2pt depth -1.6pt width 23pt, {\em Vlasov-poisson
  equations in weighted sobolev space {${W}^{m, p}(w)$} space}, Cubo, 24
  (2022), pp.~211--226.

\bibitem{KP2}
{\sc T.~Kato and G.~Ponce}, {\em On nonstationary flows of viscous and ideal
  fluids in {$L^p_s({\bf R}^2)$}}, Duke Math. J., 55 (1987), pp.~487--499.

\bibitem{KP1}
\leavevmode\vrule height 2pt depth -1.6pt width 23pt, {\em Commutator estimates
  and the {E}uler and {N}avier-{S}tokes equations}, Comm. Pure Appl. Math., 41
  (1988), pp.~891--907.

\bibitem{Loeper}
{\sc G.~Loeper}, {\em Uniqueness of the solution to the {V}lasov-{P}oisson
  system with bounded density}, J. Math. Pures Appl. (9), 86 (2006),
  pp.~68--79.

\bibitem{Lagrangian}
{\sc A.~F. Luigi~Ambrosio, Maria~Colombo}, {\em On the lagrangian structure of
  transport equations: the vlasov-poisson system}, Duke Mathematical Journal,
  166 (2017).

\bibitem{Uniqueness}
{\sc E.~Miot}, {\em A uniqueness criterion for unbounded solutions to the
  vlasov–poisson system}, Communications in Mathematical Physics (Springer
  Berlin Heidelberg), 3 (2016).

\bibitem{C.P.}
{\sc C.~Pallard}, {\em Moment propagation of weak solutions to the
  vlasov-poisson system}, Communications in Partial Diferential Equations, 37
  (2012), pp.~1273--1285.

\bibitem{K.P.}
{\sc K.~Pfaffelmoser}, {\em Global classical solutions of the vlasov-poisson
  system in three dimensions for general initial data}, Journal of Differential
  Equations, 95 (1992), pp.~282--303.

\bibitem{P.L.}
{\sc P.L.LIons and B.~Perthame}, {\em Propagation of moments and regularity for
  the 3-dimenional vlasov-poisson system}, Invent.math., 105 (1991),
  pp.~415--430.

\bibitem{Rein07}
{\sc G.~Rein}, {\em Collisionless kinetic equations from astrophysics---the
  {V}lasov-{P}oisson system}, in Handbook of differential equations:
  evolutionary equations. {V}ol. {III}, Handb. Differ. Equ.,
  Elsevier/North-Holland, Amsterdam, 2007, pp.~383--476.

\bibitem{Avg}
{\sc Y.~R.J.Diperna, P.L.Lions}, {\em ${L}_p$ regularity of velocity averages},
  Annales De Li. H. P., Section C, 8 (1991), pp.~271--287.

\bibitem{Rudin}
{\sc W.~Rudin}, {\em Functional Analysis}, International Series in Pure and
  Applied Mathematics, McGraw-Hill, 2nd~ed., 1991.

\bibitem{J.S.}
{\sc J.~Schaeffer}, {\em Global existence of smooth solutions to the
  vlasov-poisson system in three dimensions}, Communications in Partial
  Differential Equations, 16 (2009), pp.~1313--1335.

\bibitem{BV}
{\sc J.~R.~R. Simon~Labrunie, Sandrine~Marchal}, {\em Local existence and
  uniqueness of the mild solution to the 1d vlasov-poisson system with an
  initial condition of bounded variation}, Mathematical Methods in The Applied
  Sciences, 33 (2010).

\end{thebibliography}

\end{document}